\documentclass[11pt]{article}
\usepackage[english]{babel}
\usepackage[T1]{fontenc}
\usepackage[utf8x]{inputenc}
\usepackage{lmodern}
\usepackage{multirow}
\usepackage{colortbl}
\usepackage{color}
\usepackage{amsfonts}
\usepackage{MnSymbol}
\usepackage{amsthm}
\usepackage{geometry}
\usepackage{tikz}
\usepackage{algorithmic}
\usepackage[vlined]{algorithm2e}
	\newtheorem{theorem}{Theorem}
	\newtheorem{lemma}[theorem]{Lemma}
	\newtheorem{definition}[theorem]{Definition}
	\newtheorem*{property}{Property (P)}
	\newtheorem{proposition}[theorem]{Proposition}
	\newtheorem{corollary}[theorem]{Corollary}

\newcommand{\opt}{\textnormal{OPT}}
\newcommand{\inn}{\textnormal{in}}

\newcommand{\boldnoindent}[1]{\smallskip\noindent{\bf\boldmath #1.} \ }
\newcommand{\case}[1]{\smallskip\noindent{\bf\boldmath #1:} }

\definecolor{orange}{rgb}{1,0.9,0}
\definecolor{violet}{rgb}{0.8,0,1}
\definecolor{darkgreen}{rgb}{0,0.7,0}
\definecolor{grey}{rgb}{0.75,0.75,0.75}

\begin{document}

\title{\vspace*{-5mm}\Large Two-connected spanning subgraphs with at most $\frac{10}{7}\opt$ edges} 
\author{Klaus Heeger \and Jens Vygen}
\date{\small Research Institute for Discrete Mathematics, University of Bonn}

\begingroup
\makeatletter
\let\@fnsymbol\@arabic
\maketitle
\endgroup

\begin{abstract}
We present a $\frac{10}{7}$-approximation algorithm for the minimum
two-vertex-connected spanning subgraph problem.
\end{abstract}

\section{Introduction}

The two-vertex-connected spanning subgraph problem (2VC) is a fundamental problem in survivable network design.
Given a 2-vertex-connected graph, it asks for a 2-vertex-connected spanning subgraph with as few edges as possible.

In this paper, we describe a $\frac{10}{7}$-approximation algorithm for this problem.
Let $\opt(G)$ denote the minimum number of edges in any 2-connected spanning subgraph of $G$.
Then our algorithm computes a 2-connected spanning subgraph of $G$ with at most $\frac{10}{7}\opt(G)$ edges in $O(n^3)$ time, where $n=|V(G)|$.

\boldnoindent{Previous work}
2VC is NP-hard because a graph $G$ is Hamiltonian if and only if it contains a 2-vertex-connected spanning
subgraph with $n$ edges. 
Czumaj and Lingas \cite{czumaj} showed that the problem is APX-hard.
Obtaining a 2-approximation algorithm is easy, for example computing any open ear-decomposition and deleting trivial ears 
(cf.\ Section~\ref{sec:eardec}) does the job.

Khuller and Vishkin \cite{khuller} found a $\frac{5}{3}$-approximation algorithm. 
Garg, Santosh and Singla \cite{garg} improved the approximation ratio to $\frac{3}{2}$. 
The idea of both algorithms is to begin with a DFS tree, add edges to obtain 2-vertex-connectivity, and delete edges that are not needed.
There are also other $\frac{3}{2}$-approximation algorithms, e.g., by Cheriyan and Thurimella \cite{thuri} 
who compute a smallest spanning subgraph with minimum degree 1 and extend it by a minimal set of edges to be 2-vertex-connected. 

Better approximation ratios have been claimed several times:
ratio $\frac{4}{3}$ by Vempala and Vetta \cite{vempala},
ratio $\frac{5}{4}$ by Jothi, Raghavachari and Varadarajan \cite{jothi},
and ratio $\frac{9}{7}$ by Gubbala and Raghavachari \cite{gubbala}.
However, we will prove that the approach of Vempala and Vetta \cite{vempala} does not work (see the Appendix),
and Jothi, Raghavachari and Varadarajan withdrew their claim (see \cite{gubbala}).
Gubbala and Raghavachari \cite{gubbala} announced a full paper with a complete  proof, but today there is only
the approximately 70-page proof in Guballa's thesis \cite{gub} which contains some inconsistencies.
According to Raghavachari [personal communication, 2016], they are not planning anymore to revise their proof.

Apparently, the naturally arising question whether there is an approximation algorithm with ratio better than $\frac{3}{2}$ for 2VC
has been open for almost 25 years. It is also mentioned at the end of Nutov's \cite{nutov} recent survey,
which discusses generalizations such as the min-cost version and $k$-connected subgraphs.
Here we answer this question affirmatively.

\boldnoindent{Our approach}
Our work is inspired by the work of Cheriyan, Seb\H{o} and Szigeti \cite{cheriyan}, 
who proved the first ratio better than $\frac{3}{2}$ for the related 2-edge-connected spanning subgraph problem (2EC).
Our algorithm works as follows.
\begin{itemize}
\item[1.] We first delete some edges that we can identify as redundant (Section \ref{sec:redundant}). 
\item[2.] Then we compute an open ear-decomposition with special properties (Sections \ref{sec:neweardec}--\ref{sec:prop567}). 
\item[3.] Finally we delete the ``trivial'' ears (those that consist of a single edge).
\end{itemize}
We compare our result to three lower bounds (Section \ref{sec:lowerbounds}), two of which are well known lower bounds
even for 2EC. The third lower bound is new and exploits the properties obtained in steps 1 and 2.

\section{Ear-decompositions \label{sec:eardec}}

In this paper, all graphs are simple and undirected. When we say 2-connected, we mean 2-vertex-connected.
By $\delta(v)$ and $\Gamma(v)$ we denote the set of incident edges and the set of neighbours of a vertex $v$, respectively;
so $|\delta(v)|=|\Gamma(v)|$ is the degree of $v$.
When we write $G-e$, $G+f$, or $G-v$, we mean deleting an edge $e$, adding an edge $f$, or deleting a vertex
$v$ and all its incident edges.
By $V(G)$ and $E(G)$ we denote the vertex set and edge set of $G$; let $n:=|V(G)|$
be the number of vertices of the given graph $G$. Obviously $\opt(G)\ge n$.

The following well-known graph-theoretic concept is due to Whitney \cite{whitney}.

\begin{definition}
An \emph{ear-decomposition} of a graph $G$ is a sequence $P_0,P_1,\ldots,P_k$ of graphs, 
where 
$P_0$ consists of a single vertex, 
$V(G)=V(P_0)\cup\cdots\cup V(P_k)$, $E(G)=E(P_1)\cup\cdots\cup E(P_k)$,
and for all $i\in\{1,\ldots,k\}$ we have:
		\begin{itemize}
			\item $P_i$ is a circuit with exactly one vertex in $V(P_0)\cup\cdots\cup V(P_{i-1})$ \emph{(closed ear)}, or
			\item $P_i$ is a path whose endpoints, but no inner vertices, belong to $V(P_0)\cup\cdots\cup V(P_{i-1})$ \emph{(open ear)}.
		\end{itemize}

A vertex in $V(P_i)\cap (V(P_0)\cup\cdots\cup V(P_{i-1}))$ is called \emph{endpoint} of $P_i$ (even if $P_i$ is closed). 
An ear has one or two endpoints; its other vertices are called \emph{inner}.
Let $\inn(P)$ denote the set of inner vertices of an ear $P$. 
If $P$ and $Q$ are ears and $p\in\inn(P)$ is an endpoint of $Q$, then $Q$ is \emph{attached} to $P$ (at $p$).

An ear-decomposition is \emph{open} if all ears except $P_1$ are open.
We call an ear of length $l$ (that is, with $l$ edges) an \emph{$l$-ear}. 
1-ears are also called \emph{trivial} ears.
An ear is called \emph{pendant} if no nontrivial ear is attached to it.
\end{definition}

We can always assume that trivial ears come at the end of the ear-decomposition.
Whitney \cite{whitney} showed that a graph is 2-connected if and only if it has an open ear-decomposition
(and a graph is 2-edge-connected if and only if it has an ear-decomposition).
Therefore, 2VC is equivalent to computing an open ear-decomposition with as many trivial ears as possible
(deleting the trivial ears yields a 2-connected spanning subgraph).
Indeed, an arbitrary open ear-decomposition yields a 2-approximation algorithm because the number of edges in nontrivial ears
is at most $2(n-1)$: for every nontrivial ear, the number of edges is at most twice the number of inner vertices.

This is tight for 2-ears, and it already shows that we need to pay special attention to 2-ears and 3-ears. 
Let $\varphi(G)$ denote the minimum number of even ears in any ear-decomposition of $G$.
The following result, based on a fundamental theorem of Frank \cite{frank}, helps us dealing with 2-ears.

\begin{proposition}[Cheriyan, Seb\H{o} and Szigeti \cite{cheriyan}]
\label{prop:cheriyan}
For any 2-connected graph $G$ one can compute an open ear-decomposition with $\varphi(G)$ even ears in $O(n^3)$ time.
\end{proposition}

Taking such an ear-decomposition and deleting all trivial ears yields a $\frac{3}{2}$-approximation for 2VC:
for every nontrivial ear, the number of edges is at most $\frac{3}{2}$ times the number of inner vertices, except for 2-ears, for which we have to add $\frac{1}{2}$;
hence we end up with at most $\frac{3}{2}(n-1)+\frac{1}{2}\varphi(G)$ edges, which is no more than $\frac{3}{2}\opt(G)$ because
$n-1+\varphi(G)$ is a lower bound (even for 2EC, see Section \ref{sec:lowerbounds}).

In order to improve upon $\frac{3}{2}$ we have to look at 3-ears.
Cheriyan, Seb\H{o} and Szigeti \cite{cheriyan} proceed by making all 3-ears pendant and 
such that inner vertices of different 3-ears are not adjacent;
they show that this yields a  $\frac{17}{12}$-approximation for 2EC (this ratio was later improved to $\frac{4}{3}$ by \cite{vygen}).
However, their ear-decomposition is not open, so we cannot use it for 2VC.
In fact, there are 2-connected graphs that do not have an open ear-decomposition with $\varphi(G)$ even ears
in which all 3-ears are pendant: see the example in Figure \ref{ex:notall3earspendant}.
Therefore we will require different properties for 3-ears in our ear-decomposition.

\begin{figure}
\begin{center}
	\begin{tikzpicture}[scale=1.3, thick]
		\node[circle, fill, inner sep=2] (a1) at (0.5,0.3) {};
		\node[ fill, inner sep=2] (a2) at (0,1) {};
		\node[circle, fill, inner sep=2] (a3) at (0,2) {};
		\node[circle, fill, inner sep=2] (a4) at (1,2) {};
		\node[fill, inner sep=2] (a5) at (1,1) {};
		\node[ fill, inner sep=2] (b1) at (2,2) {};
		\node[circle, fill, inner sep=2] (b2) at (2,1) {};
		\node[ fill, inner sep=2] (c1) at (3,2) {};
		\node[fill, inner sep=2] (c2) at (3,1) {};
		\node[ fill, inner sep=2] (d1) at (-1,2) {};
		\node[circle, fill, inner sep=2] (d2) at (-1,1) {};
		\node[ fill, inner sep=2] (e1) at (-2,2) {};
		\node[ fill, inner sep=2] (e2) at (-2,1) {};
		\draw[] (a1) to (a2) to (a3) to (a4) to (a5) to (a1);
		\draw[red] (a4) to (b1); \draw[red] (b1) to (b2); 
		\draw[red] (b2) to (a1);
		\draw[bend left, darkgreen] (a4) to (c1); \draw[darkgreen] (c1) to (c2) to (b2);
		\draw[blue] (a3) to (d1); \draw[blue] (d1) to (d2); \draw[blue] (d2) to (a1);
		\draw[bend right, brown] (a3) to (e1); \draw [brown] (e1) to (e2); \draw[brown] (e2) to (d2);
	\end{tikzpicture}
\caption{\label{ex:notall3earspendant} A graph $G$ with $\varphi(G)=0$ in which every ear-decomposition has even ears or non-pendant 3-ears.}
\end{center}
\end{figure}
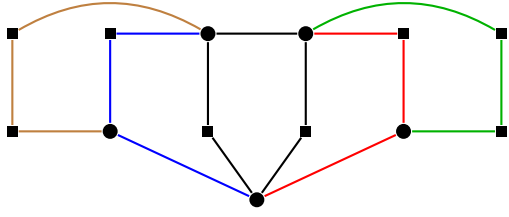

\section{Redundant edges \label{sec:redundant}}

We first need to delete some edges of our graph. 

\begin{definition}
Let $G$ be a 2-connected graph. 
An edge $e\in E(G)$ is \emph{redundant} if $G-e$ is 2-connnected and $\opt(G-e)=\opt(G)$.
\end{definition}

In other words, an edge is redundant unless it is contained in every optimum solution.
Edges incident to a vertex of degree 2 are never redundant. 
The graph in Figure \ref{ex:notall3earspendant} has no redundant edges.

Of course it is in general difficult to decide whether a certain edge is redundant. 
But in some situations we can do it, as we show now.
A slightly weaker version of the following lemma ($G-f$ 2-connected $\Rightarrow$ $f$ redundant) was already shown by Chong and Lam \cite{chong}.

\begin{lemma}
\label{lemma:redundant}
Let $G$ be a 2-connected graph, let $a,b,c,d,e \in V(G)$ be five vertices
with $\Gamma (a)=\{c,d\}$ and $\Gamma (b)=\{c,e\}$ and $f:=\{d,e\}\in E(G)$ (see Figure \ref{fig:redundant}(a)).
Then $f$ is redundant if and only if $(G-c)-f$ is connected.
\end{lemma}

\begin{figure}	
\begin{center}
	\begin{tikzpicture}[scale=0.75,thick]
	\node at (-1.5,-2) {(a)};
		\node[fill, inner sep=2, label=180:$a$] (a) at (-0.2,0.1) {};
		\node[fill, inner sep=2, label=0:$b$] (b) at (2.2,0.1) {};
		\node[circle,fill, inner sep=2, label=270:$c$] (c) at (1,-1) {};
		\node[circle,fill, inner sep=2, label=180:$d$] (d) at (0,2) {};
		\node[circle,fill, inner sep=2, label=0:$e$] (e) at (2,2) {};
		\node[label=90:$f$] at (1,1.8) {};
		\draw[red] (a) to (d);
		\draw[red] (a) to (c);
		\draw[red] (c) to (b);
		\draw[red] (e) to (b);
		\draw[dashed] (e) to (d);
	\end{tikzpicture}
	\hfill
	\begin{tikzpicture}[scale=0.75,thick]
	\node at (-1.5,-2) {(b)};
		\node[fill, inner sep=2, label=180:$a$] (a) at (-0.2,0.1) {};
		\node[fill, inner sep=2, label=0:$b$] (b) at (2.2,0.1) {};
		\node[circle,fill=black!30, inner sep=2, label=270:$c$] (c) at (1,-1) {};
		\node[circle,fill, inner sep=2, label=180:$d$] (d) at (0,2) {};
		\node[circle,fill, inner sep=2, label=0:$e$] (e) at (2,2) {};
		\node[label=90:$f$] at (1,1.8) {};
		\node[label=270:{\textcolor{darkgreen}{$g$}}] (g) at (1,-1.7) {};
		\draw (a) to (d);
		\draw[dashed, black!30] (a) to (c);
		\draw[dashed, black!30] (c) to (b);
		\draw (e) to (b);
		\draw[dashed, black!30] (e) to (d);
		\draw[densely dotted] (0.15,-2) arc (-30:30:5cm);
		\draw[densely dotted] (1.85,-2) arc (210:150:5cm);
		\draw[thick,bend right, darkgreen] (0.15,-1.6) to (1.85,-1.6);
	\end{tikzpicture}
	\hfill
	\begin{tikzpicture}[scale=0.75,thick]
	\node at (-2,-2) {(c)};
		\node[fill, inner sep=2, label=268:{$d'\!\!=\!a$}] (a) at (-0.2,0.1) {};
		\node[fill, inner sep=2, label=0:$b$] (b) at (2.2,0.1) {};
		\node[circle,fill, inner sep=2, label=270:$c$] (c) at (1,-1) {};
		\node[circle,fill, inner sep=2, label=180:$d$] (d) at (0,2) {};
		\node[circle,fill, inner sep=2, label=0:$e$] (e) at (2,2) {};
		\node[label=90:$f$] at (1,1.8) {};
		\node[circle,fill, inner sep=2, label=272:$e'$] (ee) at (-2.2,0.1) {};
		\draw[red] (a) to (d);
		\draw[red] (a) to (c);
		\draw[red] (c) to (b);
		\draw[red] (e) to (b);
		\draw (e) to (d);
		\draw[dashed] (a) to (ee);
		\node[label=90:$f'$] at (-1.2,-0.1) {};
	\end{tikzpicture}

\caption{\label{fig:redundant}Identifying some redundant edges.
Here and in all following figures, vertices shown as squares have no other incident edges than the ones shown.
In (a), $f$ is redundant by Lemma \ref{lemma:redundant}. 
In (b), $g$ is an edge connecting the two connected components of $(H-f)-c$.
In (c), deleting $f'$ exhibits $f$ as redundant.}
\end{center}
 \end{figure}
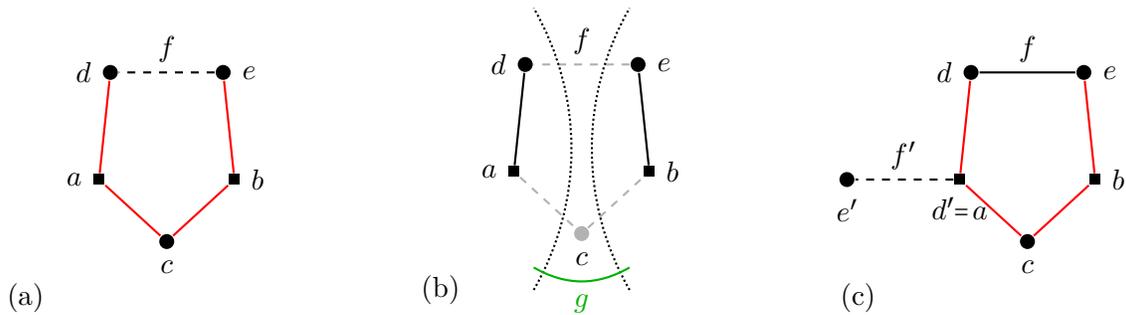

\begin{proof}
If $f$ is redundant, then $G-f$ is 2-connected, so $(G-c)-f=(G-f)-c$ is connected. 

For the other direction, let $(G-c)-f$ be connected.
Let $H$ be a 2-connected spanning subgraph of $G$ with minimum number of edges. 
If $f\notin E(H)$, we are done.

So assume $f\in E(H)$; then $H$ contains all five edges of Figure \ref{fig:redundant}(a). 
Moreover, $H-f$ is not 2-connected (due to the minimality of $H$), so let $x$ be a vertex such that $(H-f)-x$ is disconnected.
Clearly $x\in\{a,b,c\}$ because $H-x$ is connected. 
Then $(H-f)-c$ is disconnected.

Let $g\in E(G)\backslash\{f\}$ be an edge connecting the two connected components of $(H-f)-c$
(see Figure \ref{fig:redundant}(b)).
We claim that $H':=(H-f)+g$ is 2-connected, implying that $f$ is redundant.

Suppose $H'$ is not 2-connected, and  let $x$ be a vertex such that $H'-x$ is disconnected.
As above we have $x\in\{a,b,c\}$ and conclude that $H'-c$ is disconnected.
But this contradicts the choice of $g$.
\end{proof}

In the following we will often need the following property:
 
\begin{property}
A graph $G$ satisfies property (P) if it is 2-connected and
for any five vertices $a,b,c,d,e\in V(G)$ with $\Gamma (a)=\{c,d\}$ and $\Gamma (b)=\{c,e\}$ and $f=\{d,e\}\in E(G)$
we have that $f$ is not redundant.
\end{property} 

We can obtain property (P) by successively deleting edges that we can identify as redundant 
by Lemma \ref{lemma:redundant}:

\begin{corollary}
\label{cor:redundant}
Let $G$ be a 2-connected graph. Then one can compute a spanning subgraph $\bar G$ of $G$ in $O(n^3)$ time such that
$\opt(G)=\opt(\bar G)$ and $\bar G$ satisfies property (P).
\end{corollary}

\begin{proof}
Identifying all sets of five vertices $a,b,c,d,e$ with $\Gamma (a)=\{c,d\}$ und $\Gamma (b)=\{c,e\}$ and $f=\{d,e\}\in E(G)$
can be done in $O(n^2)$ time by enumerating all pairs of vertices of degree 2. Let $L$ be the set of such edges $f$.

Compute an open ear-decomposition of $G$. 
All edges of $L$ that are trivial ears are redundant and can be removed.
We now scan the $O(n)$ remaining edges in $L$ one by one. 
For each of them, we can check in $O(|E(G)|)$ time by Lemma \ref{lemma:redundant} whether it is redundant and, if so, delete it.

We claim that the resulting graph $\bar G$ satisfies property (P). Suppose not.
Let $a,b,c,d,e$ be five vertices violating property (P). 
Then $|\Gamma_G (a)|>2$ or $|\Gamma_G (b)|>2$ because otherwise $f=\{d,e\}$ would have been in $L$ and hence been deleted.
Consider the subgraph $G'$ immediately before the last time that an edge incident to $a$ or $b$ was deleted.
Say an edge $f'=\{d',e'\}$ was deleted with $d'=a$
(see Figure \ref{fig:redundant}(c)).
Then in particular there was a vertex $a'$ with $\Gamma(a')=\{c',d'\}$.
Then $a'$ is a neighbour of $d'=a$ that is different from $e'$, so $a'=d$ or $a'=c$.
But note that $a'=d$ is impossible because if $d$ has degree two, $f=\{d,e\}$ cannot be redundant.
So $a'=c$, and hence $c'=b$ and $b'=e$. Again this is impossible because if $e$ has degree two,  $f=\{d,e\}$ cannot be redundant.
\end{proof}

Hence, in the following, we may assume property (P).

\section{A new ear-decomposition \label{sec:neweardec}}

For a graph $G$ with property (P) we will compute an ear-decomposition with properties (E1)--(E7) below.
The edges in the nontrivial ears of this ear-decomposition will constitute a $\frac{10}{7}$-approximation for 2VC.
The first two of the following properties are the same as in \cite{cheriyan} (except that our ear-decomposition is open);
the other ones deal with nonpendant 3-ears.

\begin{enumerate}
\item[(E1)] The ear-decomposition is open and has $\varphi(G)$ even ears. 
\item[(E2)] Inner vertices of different pendant 3-ears are not adjacent in $G$.
\end{enumerate}
Moreover, for every nonpendant 3-ear $P$ and the first nontrivial ear $Q$ attached to $P$, say at $v$, where $E(P)=\{\{x, v\}, \{v,w\}, \{w, y\}\}$:
\begin{enumerate}
\item[(E3)] The other endpoint of $Q$ is $y$, and $y\not=x$.
\item[(E4)] If only pendant 3-ears and trivial ears are attached to $P$, and $E(Q)=\{\{v, v'\}, \{v',w'\},$ $\{w', y\}\}$, then
$w'$ has degree 2 or $\bigl( \Gamma (w')=\{y,v',v\} \text{ and } \Gamma (v')\subseteq\{v, w', y\} \bigr)$.
\item[(E5)] If $Q$ is a pendant 3-ear, then only pendant 3-ears and trivial ears are attached to $P$.
\item[(E6)] The vertex $w$ has degree 2 in $G$.
\item[(E7)] If $Q$ is a 2-ear, then the inner vertex of $Q$ has degree 2 in $G$.
\end{enumerate}

\begin{figure}	
\begin{center}
	\begin{tikzpicture}[thick,scale=0.9]
		\draw (-1,0) node {(a)};
		\node[circle, fill, inner sep=2, label=180:$x$] (x) at (0,0) {};
		\node[circle, fill=red, inner sep=2, label=180:$v$] (v) at (0,1) {};
		\node[circle, fill=red, inner sep=2, label=180:$w$] (w) at (0,2) {};
		\color{darkgreen}
		\node[label=180:$P$] (h) at (0,1.5) {};
		\color{black}
		\node[circle,fill, inner sep=2, label=180:$y$] (y) at (0,3) {};
		\draw[darkgreen] (x) to (v);
		\draw[darkgreen] (v) to (w);
		\draw[darkgreen] (w) to (y);

	\end{tikzpicture}
	\hfill
	\begin{tikzpicture}[thick,scale=0.9]
		\draw (-1,0) node {(b)};
		\node[circle, fill, inner sep=2, label=180:$y$] (a) at (0,3) {};
		\node[fill, inner sep=2, label=180:$w$] (b) at (0,2) {};
		\node[circle, fill, inner sep=2, label=180:$v$] (c) at (0,1) {};
		\node[circle, fill, inner sep=2, label=180:$x$] (d) at (0,0) {};
		\color{darkgreen}
		\node[label=180:$P$] (h) at (0,1.5) {};
		\color{black}
		\node[fill, inner sep=2] (e) at (1,2) {};
		\draw[darkgreen] (a) to (b);
		\draw[darkgreen] (b) to (c);
		\draw[darkgreen] (c) to (d);
		\draw[thick,scale=0.9, red] (a) to (e) to (c);
	\end{tikzpicture}
	\hfill
	\begin{tikzpicture}[thick,scale=0.9]
		\draw (-1,0) node {(c$'$)};
		\node[circle, fill, inner sep=2, label=180:$y$] (a) at (0,3) {};
		\node[fill, inner sep=2, label=180:$w$] (b) at (0,2) {};
		\node[circle, fill, inner sep=2, label=180:$v$] (c) at (0,1) {};
		\node[circle, fill, inner sep=2, label=180:$x$] (d) at (0,0) {};
		\color{darkgreen}
		\node[label=180:$P$] (h) at (0,1.5) {};
		\color{black}
		\node[circle, fill, inner sep=2] (e) at (1,1) {};
		\node[fill, inner sep=2] (f) at (1,2) {};
		\draw[darkgreen] (a) to (b);
		\draw[darkgreen] (b) to (c);
		\draw[darkgreen] (c) to (d);
		\draw[thick,scale=0.9, red] (a) to (f) to (e) to (c);
	\end{tikzpicture}
	\hfill
	\begin{tikzpicture}[thick,scale=0.9]
		\draw (-1.1,0) node {(c$''$)};
		\node[circle, fill, inner sep=2, label=180:$y$] (a) at (0,3) {};
		\node[fill, inner sep=2, label=180:$w$] (b) at (0,2) {};
		\node[circle, fill, inner sep=2, label=180:$v$] (c) at (0,1) {};
		\node[circle, fill, inner sep=2, label=180:$x$] (d) at (0,0) {};
		\color{darkgreen}
		\node[label=180:$P$] (h) at (0,1.5) {};
		\color{black}
		\node[fill, inner sep=2] (e) at (1,1) {};
		\node[circle, fill, inner sep=2] (f) at (1,2) {};
		\draw[darkgreen] (a) to (b);
		\draw[darkgreen] (b) to (c);
		\draw[darkgreen] (c) to (d);
		\draw[thick,scale=0.9, red] (a) to (f) to (e) to (c);
	\end{tikzpicture}
	\hfill
	\begin{tikzpicture}[thick,scale=0.9]
		\draw (-1.2,0) node {(c$'''$)};
		\node[circle, fill, inner sep=2, label=180:$x$] (x) at (0,0) {};
		\node[circle, fill, inner sep=2, label=180:$v$] (v) at (0,1) {};
		\node[fill, inner sep=2, label=180:$w$] (w) at (0,2) {};
		\node[circle,fill, inner sep=2, label=180:$y$] (y) at (0,3) {};
		\node[circle,fill=red, inner sep=2, label=0:$v'$] (v') at (1,1) {};
		\node[fill=red, inner sep=2, label=0:$w'$] (w') at (1,2) {};
		\node[circle,fill, inner sep=2] (y') at (y) {};
		\color{darkgreen}
		\node[label=180:$P$] (h) at (0,1.5) {};
		\color{black}
		\draw[darkgreen] (v) to (x);
		\draw[darkgreen] (v) to (w);
		\draw[red] (y) to (w');
		\draw[red] (v') to (v);
		\draw[red] (w') to (v');
		\draw[darkgreen] (w) to (y');
	\end{tikzpicture}
	\\[5mm]
	\begin{tikzpicture}[thick,scale=0.9]
		\draw (-1.3,0) node {(c$''''$)};
		\node[circle, fill, inner sep=2, label=180:$x$] (x) at (0,0) {};
		\node[circle, fill, inner sep=2, label=180:$v$] (v) at (0,1) {};
		\node[fill, inner sep=2, label=180:$w$] (w) at (0,2) {};
		\node[circle,fill, inner sep=2, label=180:$y$] (y) at (0,3) {};
		\node[fill=red, inner sep=2, label=0:$v'$] (v') at (1,1) {};
		\node[fill=red, inner sep=2, label=0:$w'$] (w') at (1,2) {};
		\color{darkgreen}
		\node[label=180:$P$] (h) at (0,1.5) {};
		\color{black}
		\node[circle,fill, inner sep=2] (y') at (y) {};
		\draw[thick, darkgreen] (v) to (x);
		\draw[thick, darkgreen] (v) to (w);
		\draw[thick, red] (y) to (w');
		\draw[thick, red] (v') to (v);
		\draw[thick, red] (w') to (v');
		\draw[thick, darkgreen] (w) to (y');
		\draw[thick] (w') to (v);
	\end{tikzpicture}
	\hfill\hspace*{-3mm}
	\begin{tikzpicture}[thick,scale=0.9]
		\draw (-1.4,0) node {(c$'''''$)};
		\node[circle, fill, inner sep=2, label=180:$x$] (x) at (0,0) {};
		\node[circle, fill, inner sep=2, label=180:$v$] (v) at (0,1) {};
		\node[fill, inner sep=2, label=180:$w$] (w) at (0,2) {};
		\node[circle,fill, inner sep=2, label=180:$y$] (y) at (0,3) {};
		\node[fill=red, inner sep=2, label=0:$v'$] (v') at (1,1) {};
		\node[fill=red, inner sep=2, label=0:$w'$] (w') at (1,2) {};
		\color{darkgreen}
		\node[label=180:$P$] (h) at (0,1.5) {};
		\color{black}
		\node[circle,fill, inner sep=2] (y') at (y) {};
		\draw[thick, darkgreen] (v) to (x);
		\draw[thick, darkgreen] (v) to (w);
		\draw[thick, red] (y) to (w');
		\draw[thick, red] (v') to (v);
		\draw[thick, red] (w') to (v');
		\draw[thick] (y') to (v');
		\draw[thick, darkgreen] (w) to (y');
		\draw[thick] (w') to (v);
	\end{tikzpicture}
	\hfill
	\begin{tikzpicture}[thick,scale=0.9]
		\draw (-1,0) node {(d)};
		\node[circle, fill, inner sep=2, label=180:$y$] (a) at (0,3) {};
		\node[fill, inner sep=2, label=180:$w$] (b) at (0,2) {};
		\node[circle, fill, inner sep=2, label=180:$v$] (c) at (0,1) {};
		\node[circle, fill, inner sep=2, label=180:$x$] (d) at (0,0) {};
		\color{darkgreen}
		\node[label=180:$P$] (h) at (0,1.5) {};
		\color{black}
		\draw[darkgreen] (a) to (b);
		\draw[darkgreen] (b) to (c);
		\draw[darkgreen] (c) to (d);
		\draw[thick, red] (a) to (0.9,3) to (1.4,2) to (0.9,1) to (c);
	\end{tikzpicture}
	\hfill
	\begin{tikzpicture}[thick,scale=0.9]
		\draw (-1,0) node {(e)};
		\node[circle, fill, inner sep=2, label=180:$y$] (a) at (0,3) {};
		\node[fill, inner sep=2, label=180:$w$] (b) at (0,2) {};
		\node[circle, fill, inner sep=2, label=180:$v$] (c) at (0,1) {};
		\node[circle, fill, inner sep=2, label=180:$x$] (d) at (0,0) {};
		\color{darkgreen}
		\node[label=180:$P$] (h) at (0,1.5) {};
		\color{black}
		\draw[darkgreen] (a) to (b);
		\draw[darkgreen] (b) to (c);
		\draw[darkgreen] (c) to (d);
		\draw[thick, red] (a) to (0.8,3) to (1.4,2.5) to (1.4,1.5) to (0.8,1) to (c);
	\end{tikzpicture}
	
\caption{\label{fig:3earcases}Possible 3-ears in an ear-decomposition with properties (E1)--(E7):
(a) $P$ is pendant; otherwise the first ear $Q$ (red) attached to $P$ (green) is 
(b) a 2-ear with inner vertex of degree 2, 
(c$'$) and (c$''$): a non-pendant 3-ear, 
(c$'''$) and (c$''''$) and (c$'''''$): a pendant 3-ear, 
(d): a 4-ear,
(e): an ear of length at least 5.
Again, vertices shown as squares have no other incident edges than shown.
Inner vertices of pendant ears are shown in red.}
	
\end{center}
 \end{figure}
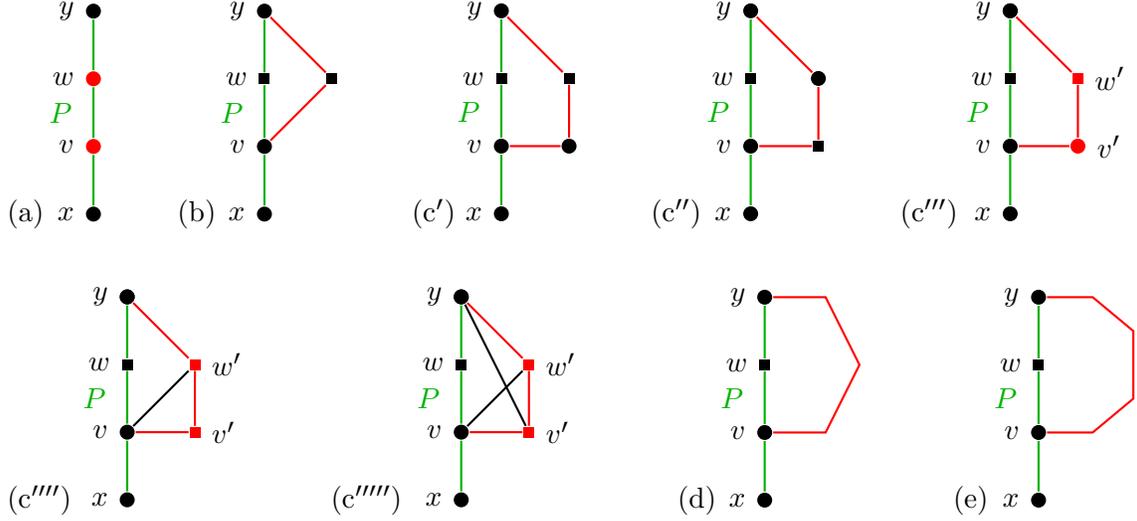

Conditions (E3)--(E7) say that every 3-ear has one of the properties shown in Figure \ref{fig:3earcases}.
The first ear can be a 3-ear only if it is pendant, i.e.\ only if $n=3$.

In the following we show how to compute such an ear-decomposition (if the graph has property (P)).
We get condition (E1) from Proposition \ref{prop:cheriyan} and maintain it throughout.
In the next section we deal with conditions (E2)--(E4). If one of them is violated, we can modify the ear-decomposition 
and increase the number of trivial ears.
Then, in Section \ref{sec:prop567}, we deal with the other conditions.

In any ear-decomposition there are $O(n)$ nontrivial ears, so we can check condition (E3)--(E7) in $O(n)$ time,
and condition (E2) in $O(n^2)$ time.

\section{How to obtain properties (E2), (E3), and (E4)}
	
We first deal with condition (E2).

\begin{lemma}
\label{lemma:prop2}
Let $G$ be a graph with property (P), and let an ear-decomposition of $G$ be given that
satisfies (E1) but not (E2).
Then we can compute in $O(n^2)$ time an ear-decomposition that satisfies (E1) and has more trivial ears.
\end{lemma}

\begin{proof}
Let $P_i$ and $P_j$ be pendant 3-ears with 
$E(P_i)=\{\{x, v\}, \{v,w\}, \{w, y\}\}$ and $E(P_j)= \{\{x', v'\}, \{v',w'\}, \{w', y'\}\}$ and 
$e=\{v,v'\}\in E(G)$. 

We consider several cases.
In each case we construct a new open pendant 5-ear $P'$ which can be added to the end of the ear-decomposition, before the trivial ears.
$P_i$ and $P_j$ will be removed, and the number of trivial ears increases by 1.

\case{Case 1} $y\neq y'$.\\
Let $P'$ be the 5-ear consisting of the $y$-$v$-path in $P_i$, the edge $\{v,v'\}$, and the $v'$-$y'$-path in $P_j$ (Figure \ref{fig:prop2}(a)).

\case{Case 2} $y=y'$ and $|\delta (w)|=|\delta (w')|=2$.\\
As $e$ is a trivial ear in the ear-decomposition, $G-e$ is 2-connected, so $(G-e)-y$ is connected. 
By Lemma \ref{lemma:redundant} this means that $e$ is redundant, violating property (P). 
Therefore this case cannot happen (Figure \ref{fig:prop2}(b)).

\case{Case 3} $y=y'$ and $|\delta (w)|>2$.\\
		\bfseries Case 3.1: \mdseries  There is a $z\in \Gamma (w)\backslash (V(P_i)\cup V(P_j))$:\\
		Let $P'$ be a 5-ear with edges $\{z,w\}, \{w, v\}, \{v, v'\}, \{v', w'\},$ and $\{w', y\}$ (Figure \ref{fig:prop2}(c)).\\
		\bfseries Case 3.2: \mdseries $x\in \Gamma (w)$:\\
		Let $P'$ be a 5-ear with edges $\{x, w\}, \{w, v\}, \{v, v'\}, \{v', w'\},$ and $\{w', y\}$ (Figure \ref{fig:prop2}(d)).\\
		\bfseries Case 3.3: \mdseries $x'\in \Gamma (w)$:\\
		Let $P'$ be a 5-ear with edges $\{x', w\}, \{w, v\}, \{v, v'\}, \{v', w'\},$ and $\{w', y\}$ (Figure \ref{fig:prop2}(e)).\\
		\bfseries Case 3.4: \mdseries $v'\in \Gamma (w)$:\\
		Let $P'$ be a 5-ear with edges $\{x, v\}, \{v, w\}, \{w, v'\}, \{v', w'\},$ and $\{w', y\}$ (Figure \ref{fig:prop2}(f)).\\
		\bfseries Case 3.5: \mdseries $w'\in \Gamma (w)$:\\
		Let $P'$ be a 5-ear with edges $\{x, v\}, \{v, v'\}, \{v', w'\}, \{w', w\},$ and $\{w, y\}$ (Figure \ref{fig:prop2}(g)).
		
\case{Case 4} $y=y'$ and $|\delta (w')|>2$. Symmetric to Case 3.
\end{proof}

\begin{figure}
\begin{center}

	\hfill
	\begin{tikzpicture}[thick,scale=0.8]
		\draw (-1,0) node {(a)};
		\node[circle,fill, inner sep=2, label=180:$x$] (x) at (0,0) {};
		\node[circle,fill=red, inner sep=2, label=180:$v$] (v) at (0,1) {};
		\node[circle,fill=red, inner sep=2, label=180:$w$] (w) at (0,2) {};
		\node[circle,fill, inner sep=2, label=180:$y$] (y) at (0,3) {};
		\node[circle,fill, inner sep=2, label=0:$x'$] (x') at (2,0) {};
		\node[circle,fill=red, inner sep=2, label=0:$v'$] (v') at (2,1) {};
		\node[circle,fill=red, inner sep=2, label=0:$w'$] (w') at (2,2) {};
		\node[circle,fill, inner sep=2, label=0:$y'$] (y') at (2,3) {};
		\color{darkgreen}
		\color{blue}
		\color{black}
		\draw[thick, darkgreen, dashed] (v) to (x);
		\draw[thick, darkgreen] (v) to (w);
		\draw[thick, darkgreen] (y) to (w);
		\draw[thick] (v') to node[above] {$e$} (v);
		\draw[thick, blue] (w') to (v');
		\draw[thick, dashed, blue] (x') to (v');
		\draw[thick, blue] (w') to (y');
	\end{tikzpicture}
	\hfill
	\begin{tikzpicture}[thick,scale=0.8]
		\draw (-1,0) node {(b)};
		\node[circle,fill, inner sep=2, label=180:$x$] (x) at (0,0) {};
		\node[circle,fill=red, inner sep=2, label=180:$v$] (v) at (0,1) {};
		\node[fill=red, inner sep=2, label=180:$w$] (w) at (0,2) {};
		\node[circle,fill, inner sep=2, label=180:$y$] (y) at (0,3) {};
		\node[circle,fill, inner sep=2, label=0:$x'$] (x') at (2,0) {};
		\node[circle,fill=red, inner sep=2, label=0:$v'$] (v') at (2,1) {};
		\node[fill=red, inner sep=2, label=0:$w'$] (w') at (2,2) {};
		\color{darkgreen}
		\color{blue}
		\color{black}
		\node[circle,fill, inner sep=2] (y') at (y) {};
		\draw[thick, darkgreen] (v) to (x);
		\draw[thick, darkgreen] (v) to (w);
		\draw[thick, darkgreen] (y) to (w);
		\draw[thick] (v') to node[above] {$e$} (v);
		\draw[thick, blue] (w') to (v');
		\draw[thick, blue] (x') to (v');
		\draw[thick, blue] (w') to (y');
	\end{tikzpicture}
	\hfill
	\begin{tikzpicture}[thick,scale=0.8]
		\draw (-1,0) node {(c)};
		\node[circle,fill, inner sep=2, label=180:$x$] (x) at (0,0) {};
		\node[circle,fill=red, inner sep=2, label=180:$v$] (v) at (0,1) {};
		\node[circle,fill=red, inner sep=2, label=0:$w$] (w) at (0,2) {};
		\node[circle,fill, inner sep=2, label=180:$y$] (y) at (0,3) {};
		\node[circle,fill, inner sep=2, label=0:$x'$] (x') at (2,0) {};
		\node[circle,fill=red, inner sep=2, label=0:$v'$] (v') at (2,1) {};
		\node[circle,fill=red, inner sep=2, label=0:$w'$] (w') at (2,2) {};
		\node[circle,fill, inner sep=2] (y') at (y) {};
		\node[circle, fill, inner sep=2, label=180:$z$] (z) at (-1,2) {};
		\color{darkgreen}
		\color{blue}
		\color{black}
		\draw[thick, dashed, darkgreen] (v) to (x);
		\draw[thick, darkgreen] (v) to (w);
		\draw[thick, dashed, darkgreen] (y) to (w);
		\draw[thick] (v') to (v);
		\draw[thick, blue] (w') to (v');
		\draw[thick, dashed, blue] (x') to (v');
		\draw[thick, blue] (w') to (y');
		\draw[thick] (w) to (z);
	\end{tikzpicture}
	\hfill\,\\[5mm]
	\begin{tikzpicture}[thick,scale=0.8]
		\draw (-1,0) node {(d)};
		\node[circle,fill, inner sep=2, label=180:$x$] (x) at (0,0) {};
		\node[circle,fill=red, inner sep=2, label=-45:$v$] (v) at (0,1) {};
		\node[circle,fill=red, inner sep=2, label=180:$w$] (w) at (0,2) {};
		\node[circle,fill, inner sep=2, label=180:$y$] (y) at (0,3) {};
		\node[circle,fill, inner sep=2, label=0:$x'$] (x') at (2,0) {};
		\node[circle,fill=red, inner sep=2, label=0:$v'$] (v') at (2,1) {};
		\node[circle,fill=red, inner sep=2, label=0:$w'$] (w') at (2,2) {};
		\color{darkgreen}
		\color{blue}
		\color{black}
		\node[circle,fill, inner sep=2] (y') at (y) {};
		\draw[thick, dashed, darkgreen] (v) to (x);
		\draw[thick, darkgreen] (v) to (w);
		\draw[thick, dashed, darkgreen] (y) to (w);
		\draw[thick] (v') to (v);
		\draw[thick, blue] (w') to (v');
		\draw[thick, dashed, blue] (x') to (v');
		\draw[thick, blue] (w') to (y');
		\draw[thick, bend right] (w) to (x);
	\end{tikzpicture}
	\hfill
	\begin{tikzpicture}[thick,scale=0.8]
		\draw (-1,0) node {(e)};
		\node[circle,fill, inner sep=2, label=180:$x$] (x) at (0,0) {};
		\node[circle,fill=red, inner sep=2, label=180:$v$] (v) at (0,1) {};
		\node[circle,fill=red, inner sep=2, label=180:$w$] (w) at (0,2) {};
		\node[circle,fill, inner sep=2, label=180:$y$] (y) at (0,3) {};
		\node[circle,fill, inner sep=2, label=0:$x'$] (x') at (2,0) {};
		\node[circle,fill=red, inner sep=2, label=0:$v'$] (v') at (2,1) {};
		\node[circle,fill=red, inner sep=2, label=0:$w'$] (w') at (2,2) {};
		\color{darkgreen}
		\color{blue}
		\color{black}
		\node[circle,fill, inner sep=2] (y') at (y) {};
		\draw[thick, dashed, darkgreen] (v) to (x);
		\draw[thick, darkgreen] (v) to (w);
		\draw[thick, dashed, darkgreen] (y) to (w);
		\draw[thick] (v') to (v);
		\draw[thick, blue] (w') to (v');
		\draw[thick, dashed, blue] (x') to (v');
		\draw[thick, blue] (w') to (y');
		\draw[thick] (w) to (x');
	\end{tikzpicture}
	\hfill
	\begin{tikzpicture}[thick,scale=0.8]
		\draw (-1,0) node {(f)};
		\node[circle,fill, inner sep=2, label=180:$x$] (x) at (0,0) {};
		\node[circle,fill=red, inner sep=2, label=180:$v$] (v) at (0,1) {};
		\node[circle,fill=red, inner sep=2, label=180:$w$] (w) at (0,2) {};
		\node[circle,fill, inner sep=2, label=180:$y$] (y) at (0,3) {};
		\node[circle,fill, inner sep=2, label=0:$x'$] (x') at (2,0) {};
		\node[circle,fill=red, inner sep=2, label=0:$v'$] (v') at (2,1) {};
		\node[circle,fill=red, inner sep=2, label=0:$w'$] (w') at (2,2) {};
		\node[circle,fill, inner sep=2] (y') at (y) {};
		\color{darkgreen}
		\color{blue}
		\color{black}
		\draw[thick, darkgreen] (v) to (x);
		\draw[thick, darkgreen] (v) to (w);
		\draw[thick, dashed, darkgreen] (y) to (w);
		\draw[thick, dashed] (v') to (v);
		\draw[thick, blue] (w') to (v');
		\draw[thick, dashed, blue] (x') to (v');
		\draw[thick, blue] (w') to (y');
		\draw[thick] (w) to (v');
	\end{tikzpicture}
	\hfill
	\begin{tikzpicture}[thick,scale=0.8]
		\draw (-1,0) node {(g)};
		\node[circle,fill, inner sep=2, label=180:$x$] (x) at (0,0) {};
		\node[circle,fill=red, inner sep=2, label=180:$v$] (v) at (0,1) {};
		\node[circle,fill=red, inner sep=2, label=180:$w$] (w) at (0,2) {};
		\node[circle,fill, inner sep=2, label=180:$y$] (y) at (0,3) {};
		\node[circle,fill, inner sep=2, label=0:$x'$] (x') at (2,0) {};
		\node[circle,fill=red, inner sep=2, label=0:$v'$] (v') at (2,1) {};
		\node[circle,fill=red, inner sep=2, label=0:$w'$] (w') at (2,2) {};
		\color{darkgreen}
		\color{blue}
		\color{black}
		\node[circle,fill, inner sep=2] (y') at (y) {};
		\draw[thick, darkgreen] (v) to (x);
		\draw[thick, dashed, darkgreen] (v) to (w);
		\draw[thick, darkgreen] (y) to (w);
		\draw[thick] (v') to (v);
		\draw[thick, blue] (w') to (v');
		\draw[thick, dashed, blue] (x') to (v');
		\draw[thick, blue, dashed] (w') to (y');
		\draw[thick] (w') to (w);
	\end{tikzpicture}

\caption{\label{fig:prop2}Proof of Lemma \ref{lemma:prop2}.
Two adjacent pendant 3-ears $P_i$ (green) and $P_j$ (blue). Black edges were old trivial ears. Dashed edges form new trivial ears.}
	
\end{center}
\end{figure}
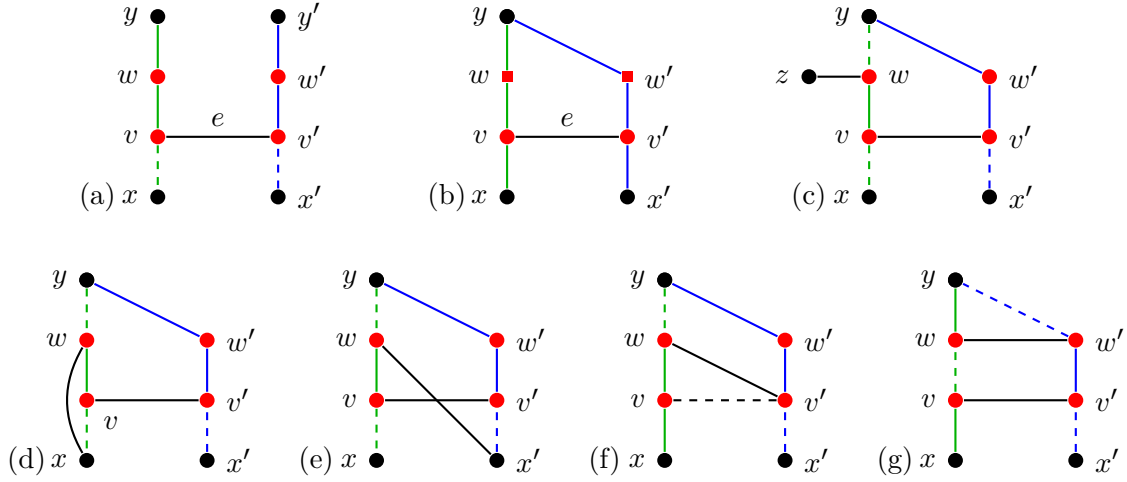

Now we consider condition (E3).
	
\begin{lemma}
\label{lemma:prop3}
Given an ear-decomposition of $G$ that satisfies (E1) but not (E3),
we can compute in $O(n)$ time an ear-decomposition that satisfies (E1) and has more trivial ears.
\end{lemma}

\begin{proof}
Let $P$ be a nonpendant 3-ear with $E(P)=\{\{x,v\}, \{v,w\}, \{w,y\}\}$ such that the first nontrivial ear $Q$ attached to $P$ 
has $v$ as an endpoint, but the other endpoint of $Q$ is not $y$.

If the other endpoint of $Q$ is $w$ or $x$ (this will always be the case if $P$ is the first (closed) ear), 
let $e$ be the edge of $P$ that connects the two endpoints of $Q$.
We can modify $P$ by replacing $e$ by $Q$. The new ear is even if and only if $Q$ is. 
The edge $e$ becomes a new trivial ear, and $Q$ vanishes
(Figure \ref{fig:prop3}, (a) and (b))

If the other endpoint of $Q$ does not belong to $V(P)$, we can replace $P$ and $Q$ by a new ear 
consisting of $Q$ and the $v$-$y$-path in $P$. This new ear is even if and only if $Q$ is, and can be put at the position of $Q$ in the
ear-decomposition. The edge $\{v,x\}$ becomes a new trivial ear (Figure \ref{fig:prop3}(c)).
\end{proof}
	
\begin{figure}
\begin{center}
	\begin{tikzpicture}[]
		\draw (-1,0) node {(a)};
		\node[circle, fill, inner sep=2, label=180:$x$] (a) at (0,0) {};
		\node[circle, fill, inner sep=2, label=180:$v$] (b) at (0,1) {};
		\node[circle,fill, inner sep=2, label=180:$w$] (c) at (0,2) {};
		\node[circle,fill, inner sep=2, label=180:$y$] (d) at (0,3) {};
		\color{red}
		\node[label=0:$Q$] (q) at (1,1.5) {};
		\color{darkgreen}
		\node[label=180:$P$] (h) at (0,1.5) {};
		\color{black}
		\draw[thick, darkgreen] (a) to (b);
		\draw[thick, darkgreen, dashed] (b) to node[right] {$e$} (c);
		\draw[thick, darkgreen] (c) to (d);
		\draw[thick, red] (b) to (1,1) to (1,2) to (c);
	\end{tikzpicture}
	\hfill
	\begin{tikzpicture}[]
		\draw (-1,0) node {(b)};	
		\node[circle, fill, inner sep=2, label=180:$x$] (a) at (0,0) {};
		\node[circle, fill, inner sep=2, label=180:$v$] (b) at (0,1) {};
		\node[circle,fill, inner sep=2, label=180:$w$] (c) at (0,2) {};
		\node[circle,fill, inner sep=2, label=180:$y$] (d) at (0,3) {};
		\color{red}
		\node[label=0:$Q$] (q) at (1,0.5) {};
		\color{darkgreen}
		\node[label=180:$P$] (h) at (0,1.5) {};
		\color{black}
		\draw[thick, darkgreen, dashed] (a) to node[right] {$e$} (b);
		\draw[thick,darkgreen] (b) to (c);
		\draw[thick, darkgreen] (c) to (d);
		\draw[thick, red] (b) to (1,1) to (1,0) to (a);
	\end{tikzpicture}
	\hfill
	\begin{tikzpicture}[]
		\draw (-1,0) node {(c)};	
		\node[circle, fill, inner sep=2, label=180:$x$] (a) at (0,0) {};
		\node[circle, fill, inner sep=2, label=180:$v$] (b) at (0,1) {};
		\node[circle,fill, inner sep=2, label=180:$w$] (c) at (0,2) {};
		\node[circle,fill, inner sep=2, label=180:$y$] (d) at (0,3) {};
		\node[circle,fill, inner sep=2] (e) at (1,2) {};
		\color{red}
		\node[label=0:$Q$] (q) at (1,1.5) {};
		\color{darkgreen}
		\node[label=180:$P$] (h) at (0,1.5) {};
		\color{black}
		\draw[thick, darkgreen, dashed] (a) to (b);
		\draw[thick,darkgreen] (b) to (c);
		\draw[thick, darkgreen] (c) to (d);
		\draw[thick, red] (b) to (1,1) to (e);
	\end{tikzpicture}
	\caption{\label{fig:prop3}Proof of Lemma \ref{lemma:prop3}. 
	The nonpendant 3-ear $P$ (green) and the first ear $Q$ attached to it (red) violate condition (E3).
	The dashed edge becomes a trivial ear in the new ear-decomposition.}

\end{center}
\end{figure}
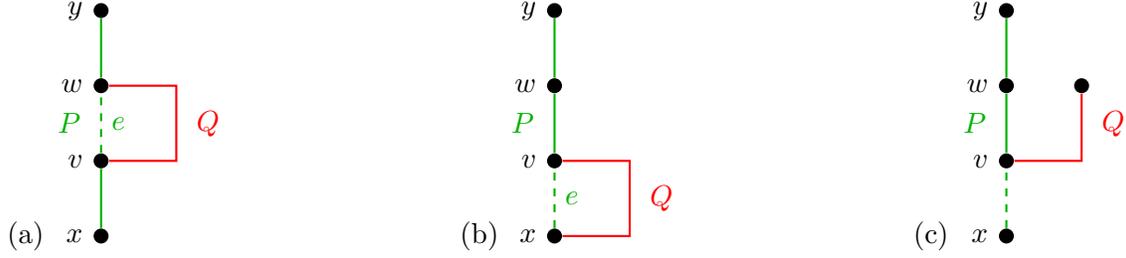	
		
Now we turn to (E4).	
	
\begin{lemma}
\label{lemma:prop4}
Let $G$ be a graph with property (P), and let an ear-decomposition of $G$ be given that
satisfies (E1), (E2) and (E3), but not (E4).
Then we can compute in $O(n)$ time an ear-decomposition that satisfies (E1) and has more trivial ears.
\end{lemma}

\begin{proof}
Let $P$ be a 3-ear with $E(P)=\{\{x,v\},\{v,w\},\{w,y\}\}$ such that only pendant 3-ears and trivial ears are attached to $P$.
Let $Q$ be the first nontrivial ear attached to it, and $E(Q)=\{\{v, v'\}, \{v',w'\}, \{w', y\}\}$. Suppose $w'$ has degree more than 2.

Then property (E4) requires $\Gamma (w')=\{y,v',v\}$ and $\Gamma (v')\subseteq\{v, w', y\}$, so suppose this does not hold.

Moreover, we may assume that after $P$ there are only the ears attached to $P$ (because these are pendant) and trivial ears in the ear-decomposition.

\case{Case 1} There exists a $z\in\Gamma (w')\backslash\{y,v',v\}$:\\
		\bfseries Case 1.1: \mdseries $z=w$:\\
		Then let $P'$ be a 5-ear with edges $\{x,v\}, \{v,v'\}, \{v',w'\}, \{w',w\},$ and $\{w,y\}$ (Figure \ref{fig:prop4}(a)).\\
		\bfseries Case 1.2: \mdseries $z\neq w$:\\
		Then let $P'$ be a 5-ear with edges $\{y,w\}, \{w,v\}, \{v,v'\}, \{v',w'\},$ and $\{w',z\}$ (Figure \ref{fig:prop4}(b)).
		
\case{Case 2} $\Gamma (w')=\{y,v',v\}$ and there exists a $z\in \Gamma (v')\backslash\{v,w',y\}$:\\
		\bfseries Case 2.1: \mdseries $z=w$:\\
		Then let $P'$ be a 5-ear with edges $\{x,v\}, \{v,w\}, \{w,v'\}, \{v',w'\},$ and $\{w',y\}$ (Figure \ref{fig:prop4}(c)).\\
		\bfseries Case 2.2: \mdseries $z\neq w$:\\
		Then let $P'$ be a 5-ear with edges $\{y,w\}, \{w,v\}, \{v,w'\}, \{w',v'\},$ and $\{v',z\}$ (Figure \ref{fig:prop4}(d)).

We replace the 3-ears $P$ and $Q$ by the new open 5-ear $P'$.
Due to property (E2), $z$ in Case 1 and Case 2 cannot be an inner vertex of a pendant 3-ear, so we can put $P'$ at the position of $P$ (or $Q$) in the ear-decomposition.
In each case there is one trivial ear more than before.
\end{proof}
	
\begin{figure}
\begin{center}	
	\begin{tikzpicture}
		\draw (-1,0) node {(a)};	
		\node[circle, fill, inner sep=2, label=180:$x$] (x) at (0,0) {};
		\node[circle, fill, inner sep=2, label=180:$v$] (v) at (0,1) {};
		\node[circle, fill, inner sep=2, label=180:$w$] (w) at (0,2) {};
		\node[circle,fill, inner sep=2, label=180:$y$] (y) at (0,3) {};
		\node[circle,fill=red, inner sep=2, label=270:$v'$] (v') at (1,1) {};
		\node[circle, fill=red, inner sep=2, label=90:$w'$] (w') at (1,2) {};
		\node[circle,fill, inner sep=2] (y') at (y) {};
		\color{darkgreen}
		\node[label=180:$P$] (h) at (0.1,1.5) {};
		\color{red}
		\node[label=180:$Q$] (h) at (1.2,1.5) {};
		\color{black}
		\draw[thick, darkgreen] (v) to (x);
		\draw[thick, darkgreen, dashed] (v) to (w);
		\draw[thick, red, dashed] (y') to (w');
		\draw[thick, red] (v') to (v);
		\draw[thick, red] (w') to (v');
		\draw[thick, darkgreen] (w) to (y);
		\draw[thick] (w) to (w');
	\end{tikzpicture}
	\hfill
	\begin{tikzpicture}
		\draw (-1,0) node {(b)};	
		\node[circle, fill, inner sep=2, label=180:$x$] (x) at (0,0) {};
		\node[circle, fill, inner sep=2, label=180:$v$] (v) at (0,1) {};
		\node[circle, fill, inner sep=2, label=180:$w$] (w) at (0,2) {};
		\color{darkgreen}
		\node[label=180:$P$] (h) at (0.1,1.5) {};
		\color{red}
		\node[label=180:$Q$] (h) at (1.2,1.5) {};
		\color{black}
		\node[circle,fill, inner sep=2, label=180:$y$] (y) at (0,3) {};
		\node[circle,fill=red, inner sep=2, label=270:$v'$] (v') at (1,1) {};
		\node[circle,fill=red, inner sep=2, label=90:$w'$] (w') at (1,2) {};
		\node[circle,fill, inner sep=2] (y') at (y) {};
		\node[circle, fill, inner sep=2, label=90:$z$] (z) at (1.7,2.7) {};
		\draw[thick, darkgreen, dashed] (v) to (x);
		\draw[thick, darkgreen] (v) to (w);
		\draw[thick, red, dashed] (y) to (w');
		\draw[thick, red] (v') to (v);
		\draw[thick, red] (w') to (v');
		\draw[thick, darkgreen] (w) to (y');
		\draw[thick] (w') to (z);
	\end{tikzpicture}
	\hfill	
	\begin{tikzpicture}
		\draw (-1,0) node {(c)};	
		\node[circle, fill, inner sep=2, label=180:$x$] (x) at (0,0) {};
		\node[circle, fill, inner sep=2, label=180:$v$] (v) at (0,1) {};
		\color{darkgreen}
		\node[label=180:$P$] (h) at (0.1,1.5) {};
		\color{red}
		\node[label=0:$Q$] (h) at (0.8,1.5) {};
		\color{black}
		\node[circle, fill, inner sep=2, label=180:$w$] (w) at (0,2) {};
		\node[circle,fill, inner sep=2, label=180:$y$] (y) at (0,3) {};
		\node[circle,fill=red, inner sep=2, label=270:$v'$] (v') at (1,1) {};
		\node[fill=red, inner sep=2, label=90:$w'$] (w') at (1,2) {};
		\node[circle,fill, inner sep=2] (y') at (y) {};
		\node[circle, fill, inner sep=2] (z) at (w) {};
		\draw[thick, darkgreen] (v) to (x);
		\draw[thick, darkgreen] (v) to (w);
		\draw[thick, red] (y) to (w');
		\draw[thick, red, dashed] (v') to (v);
		\draw[thick, red] (w') to (v');
		\draw[thick] (z) to (v');
		\draw[thick, darkgreen, dashed] (w) to (y');
		\draw[thick, dashed] (w') to (v);
	\end{tikzpicture}
	\hfill
	\begin{tikzpicture}
		\draw (-1,0) node {(d)};	
		\node[circle, fill, inner sep=2, label=180:$x$] (x) at (0,0) {};
		\node[circle, fill, inner sep=2, label=180:$v$] (v) at (0,1) {};
		\color{darkgreen}
		\node[label=180:$P$] (h) at (0.1,1.5) {};
		\color{red}
		\node[label=0:$Q$] (h) at (0.8,1.5) {};
		\color{black}
		\node[circle, fill, inner sep=2, label=180:$w$] (w) at (0,2) {};
		\node[circle,fill, inner sep=2, label=180:$y$] (y) at (0,3) {};
		\node[circle,fill=red, inner sep=2, label=270:$v'$] (v') at (1,1) {};
		\node[fill=red, inner sep=2, label=90:$w'$] (w') at (1,2) {};
		\node[circle,fill, inner sep=2] (y') at (y) {};
		\node[circle, fill, inner sep=2, label=270:$z$] (z) at (1.7,0.3) {};
		\draw[thick, darkgreen, dashed] (v) to (x);
		\draw[thick, darkgreen] (v) to (w);
		\draw[thick, red, dashed] (y) to (w');
		\draw[thick, dashed, red] (v') to (v);
		\draw[thick, red] (w') to (v');
		\draw[thick] (z) to (v');
		\draw[thick, darkgreen] (w) to (y');
		\draw[thick] (w') to (v);
\end{tikzpicture}
		\caption{\label{fig:prop4}Proof of Lemma \ref{lemma:prop4}. 
	The nonpendant 3-ear $P$ (green) and the first ear $Q$ attached to it (red) violate condition (E4).
	$Q$ is a pendant 3-ear. Black edges are trivial ears. The dashed edges become trivial ears in the new ear-decomposition.
	Again, squares indicate vertices that have no other incident edges than those shown. In (b) and (d), $z=x$ is possible.}

\end{center}
\end{figure}
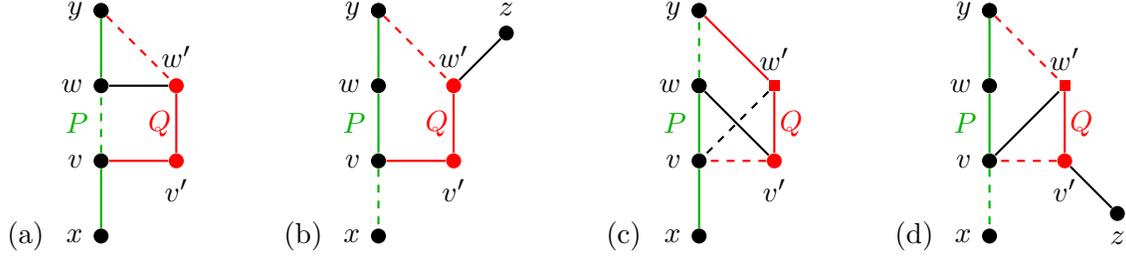

\section{How to obtain properties (E5), (E6), and (E7) \label{sec:prop567}}

In the following lemma we show how to obtain properties (E5), (E6) and (E7) simultaneously, maintaining (E1).
The number of trivial ears will not decrease.

\begin{lemma}
\label{lemma:prop567}
Given an ear-decomposition of $G$ that satisfies (E1),
we can compute in $O((j+1)n^2)$ time an ear-decomposition that satisfies (E1), (E5), (E6), and (E7),
and in which the number of trivial ears is $j$ more than before, where $j\ge 0$.
\end{lemma}

\begin{proof}
Let $P_i$ be the first nonpendant 3-ear that violates at least one of the conditions (E5), (E6), or (E7). 
We will perform changes to the ear-decomposition, maintaining (E1), such that afterwards 
the number of trivial ears increases or this number remains constant but the first
ear violating one of the conditions (E5), (E6), or (E7) has a larger index.

At the beginning and after every modification of the ear-decomposition we apply Lemma \ref{lemma:prop3} 
in order to ensure that the ear-decomposition satisfies condition (E3).

Let $Q$ be the first nontrivial ear attached to $P_i$. 
Let $E(P_i)=\{\{x,v\}, \{v,w\}, \{w,y\}\}$, where $v$ and $y$ are the endpoints of $Q$.
We will proceed with the following steps.

\bigskip
\case{Step 1: if $P_i$ violates (E5)} 
$Q$ is a pendant 3-ear, but there is a nontrivial ear $Q'$ attached to $P$ that is not a pendant 3-ear.

We can then move $Q$ and all other pendant 3-ears attached to $P$ to the position after $Q'$.
Note that this may create a new violation of (E7), but we will deal with this in Step 3.
 
\bigskip
\case{Step 2. if $P_i$ violates (E6)} 
$w$ has degree more than 2.

Let $X:=\bigcup_{j=0}^{i} V(P_j)$.
Let $R$ be the first ear attached to $P$ at $w$ (possibly trivial) and $u$ the other endpoint of $R$.
Set $S:=R$ and $a:=u$, and consider the following procedure (see Figure \ref{fig:prop567}(a) for an example):

\begin{algorithm}
		$S:=R$\;
		$a:=u$\;
		\While{$a\notin X$}{
			let $P'$ be the ear with $a\in\inn(P')$\;
			let $b$ be an endpoint of $P'\!$, if possible such that the $a$-$b$-path $T$ in $P'$ has even length\;
			$a:=b$\;
			$S:=S\cup T$\;
		}
\end{algorithm}
		
This procedure terminates because the index of $P'$ decreases in every iteration. For the same reason, $S$ is always a path;
it ends in $X\backslash\{w\}$ because $R$ is the first ear attached to $P$ at $w$.

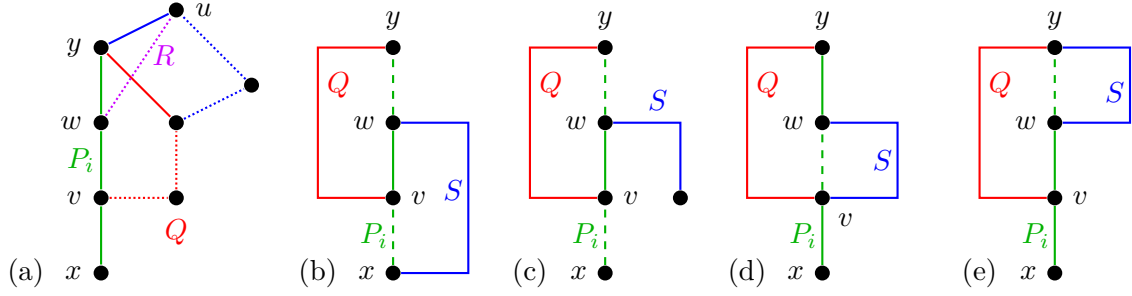
\begin{figure}
\begin{center}
		\begin{tikzpicture}[scale=1.0,thick]
			\draw (-1,0) node {(a)};	
			\node[fill, circle, inner sep=2, label=180:$x$] (x) at (0,0) {};
			\node[fill, circle, inner sep=2, label=180:$v$] (v) at (0,1) {};
			\node[fill, circle, inner sep=2, label=180:$w$] (w) at (0,2) {};
			\node[fill, circle, inner sep=2, label=180:$y$] (y) at (0,3) {};
			\node[fill, circle, inner sep=2] (a) at (1,2) {};
			\node[fill, circle, inner sep=2] (b) at (1,1) {};
			\node[fill, circle, inner sep=2, label=0:$u$] (c) at (1,3.5) {};
			\node[fill, circle, inner sep=2] (d) at (2,2.5) {};
			\color{violet}
			\node[label=0:$R$] (k) at (0.4,2.9) {};
			\color{darkgreen}
			\node[label=180:$P_i$] (h) at (0.2,1.5) {};
			\color{red}
			\node[label=270:$Q$] (i) at (1,1) {};
			\color{black}
			\draw[thick, darkgreen] (x) to (v) to (w) to (y);
			\draw[thick, red, densely dotted] (v) to (b) to (a);
			\draw[thick, red] (a) to (y);
			\draw[thick, blue, densely dotted] (a) to (d) to (c);
			\draw[thick, blue] (c) to (y);
			\draw[thick, violet, densely dotted] (w) to (c);
		\end{tikzpicture}
		\hfill
		\begin{tikzpicture}[]
			\draw (-1,0) node {(b)};	
			\node[circle, fill, inner sep=2, label=180:$x$] (a) at (0,0) {};
			\node[circle, fill, inner sep=2, label=0:$v$] (b) at (0,1) {};
			\node[circle,fill, inner sep=2, label=180:$w$] (c) at (0,2) {};
			\node[circle,fill, inner sep=2, label=90:$y$] (d) at (0,3) {};
			\color{red}
			\node[label=180:$Q$] (q) at (-0.3,2.5) {};
			\color{blue}
			\node[label=0:$S$] (e) at (0.4,1.1) {};
			\color{darkgreen}
			\node[label=180:$P_i$] (h) at (0.2,0.5) {};
			\color{black}
			\draw[thick, dashed, darkgreen] (a) to (b);
			\draw[thick, blue] (a) to (1,0) to (1,2) to (c);
			\draw[thick, darkgreen] (b) to (c);
			\draw[thick, dashed, darkgreen] (c) to (d);
			\draw[thick, red] (b) to (-1,1) to (-1,3) to (d);
		\end{tikzpicture}
		\hfill
		\begin{tikzpicture}[]
			\draw (-1,0) node {(c)};	
			\node[circle, fill, inner sep=2, label=180:$x$] (a) at (0,0) {};
			\node[circle, fill, inner sep=2, label=0:$v$] (b) at (0,1) {};
			\node[circle,fill, inner sep=2, label=180:$w$] (c) at (0,2) {};
			\node[circle,fill, inner sep=2, label=90:$y$] (d) at (0,3) {};
			\node[circle,fill, inner sep=2] (e) at (1,1) {};
			\color{darkgreen}
			\node[label=180:$P_i$] (h) at (0.2,0.5) {};
			\color{blue}
			\node[label=90:$S$] (q) at (0.7,1.9) {};
			\color{red}
			\node[label=180:$Q$] (s) at (-0.3,2.5) {};
			\color{black}
			\draw[thick, darkgreen, dashed] (a) to (b);
			\draw[thick, darkgreen] (b) to (c);
			\draw[thick, dashed, darkgreen] (c) to (d);
			\draw[thick, blue] (c) to (1,2) to(e);
			\draw[thick, red] (b) to (-1,1) to (-1,3)  to (d);
		\end{tikzpicture}
		\hfill
		\begin{tikzpicture}[]
			\draw (-1,0) node {(d)};	
			\node[circle, fill, inner sep=2, label=180:$x$] (a) at (0,0) {};
			\node[circle, fill, inner sep=2, label=315:$v$] (b) at (0,1) {};
			\node[circle,fill, inner sep=2, label=180:$w$] (c) at (0,2) {};
			\node[circle,fill, inner sep=2, label=90:$y$] (d) at (0,3) {};
			\color{blue}
			\node[label=180:$S$] (q) at (1.2,1.5) {};
			\color{darkgreen}
			\node[label=180:$P_i$] (h) at (0.2,0.5) {};
			\color{red}
			\node[label=180:$Q$] (s) at (-0.3,2.5) {};
			\color{black}
			\draw[thick, darkgreen] (a) to (b);
			\draw[thick, darkgreen, dashed] (b) to (c);
			\draw[thick, darkgreen] (c) to (d);
			\draw[thick, blue] (b) to (1,1) to (1,2) to (c);
			\draw[thick, red] (b) to (-1,1) to (-1,3) to (d);
		\end{tikzpicture}
		\hfill
		\begin{tikzpicture}[]
			\draw (-1,0) node {(e)};	
			\node[circle, fill, inner sep=2, label=180:$x$] (a) at (0,0) {};
			\node[circle, fill, inner sep=2, label=0:$v$] (b) at (0,1) {};
			\node[circle,fill, inner sep=2, label=180:$w$] (c) at (0,2) {};
			\node[circle,fill, inner sep=2, label=90:$y$] (d) at (0,3) {};
			\color{blue}
			\node[label=90:$S$] (q) at (0.8,2) {};
			\color{red}
			\node[label=180:$Q$] (s) at (-0.3,2.5) {};
			\color{darkgreen}
			\node[label=180:$P_i$] (h) at (0.2,0.5) {};
			\color{black}
			\draw[thick, darkgreen] (a) to (b);
			\draw[thick, darkgreen] (b) to (c);
			\draw[thick, darkgreen, dashed] (c) to (d);
			\draw[thick, blue] (c) to (1,2) to (1,3) to (d);
			\draw[thick, red] (b) to (-1,1) to (-1,3) to (d);
		\end{tikzpicture}		
		
\caption{\label{fig:prop567}Proof of Lemma \ref{lemma:prop567}: Condition (E6) is violated. $P_i$ (green) is a nonpendant 3-ear,
$Q$ (red) is the first ear attached to it (at $v$), and an ear $R$ has endpoint $w$.	
(a) Example of the the computation of $S$ (dotted edges), with $R$ in purple and another ear in blue.
In (b)--(e), $S$ is shown in blue, and the resulting trivial ears are dashed.}

\end{center}
\end{figure}

In the following we will create a new open ear that contains all of $S$.
Every ear (except $R$) that was used for the construction of $S$, i.e., that played the role of $P'$ in the above procedure,
has a part in $S$ and another part outside $S$. 
We will remove the edges in $S$ from these ears.
If these ears were odd, they remain odd.	
Note that $Q$ can be one of these ears (then it is the last one), as in Figure \ref{fig:prop567}(a).
		
\case{Case 1} $S$ ends in $X\backslash\{v, y\}$.

Then $Q$ and $S$ are disjoint. Replace $P_i$ and $Q$ by an ear that consists of $S$, $\{w,v\}$ and $Q$ 
(see Figure \ref{fig:prop567}, (b) and (c)). Remove the edges in $S$ from all other ears.
The edges $\{y,w\}$ and $\{v,x\}$ become trivial ears.

\case{Case 2} $S$ ends in $v$.

Then $S$ has more than one edge. 
Replace $P_i$ by an ear that consists of $\{x,v\}$, $S$, and $\{w,y\}$ 
(see Figure \ref{fig:prop567}(d), but again note that $S$ can contain part of $Q$ like in Figure \ref{fig:prop567}(a)). 
Remove the edges in $S$ from all other ears.
The edge $\{v,w\}$ becomes a trivial ear.

\case{Case 3} $S$ ends in $y$.

Then again $S$ has more than one edge. 
Replace $P_i$ by an ear that consists of $\{x,v\}, \{v,w\}$, and $S$ 
(see Figure \ref{fig:prop567}(e), but note that $S$ can contain part of $Q$). Remove the edges in $S$ from all other ears.
The edge $\{w,y\}$ becomes a trivial ear.

Note that in each of the three cases the new ear is open and has more than three edges and thus does not violate conditions (E5), (E6), or (E7).
We can put it at the position of $P_i$.
However, it may be attached to an ear $P_h$ with $h<i$ to which also an earlier pendant 3-ear is attached, now violating (E5).
In this case we immediately apply Step 1 to $P_h$.

We may also have created new 3-ears that violate one of the conditions (E5), (E6), or (E7), but they come later in the ear-decomposition.
The only possible new even ear is the one we designed from $S$. 
But then $R$ was even or one of the ears part of which belongs to $S$ changes from even to odd. So (E1) is maintained.

The (possiby trivial) ear $R$ vanishes, but at least one edge of $P_i$ becomes a new trivial ear,
so the number of trivial ears does not decrease.

\bigskip
\case{Step 3.  if $P_i$ satsifies (E6) but violates (E7)} 
$Q$ is a 2-ear with inner vertex $v'$ of degree at least three.

Replace $P_i$ and $Q$ by the ears $P'$ and $Q'$, where $E(P')=\{\{y,v'\}, \{v',v\}, \{v,x\}\}$ and $E(Q')=\{\{v,w\}, \{w,y\}\}$ 
(see Figure \ref{fig:prop567b}). This will violate condition (E6). Now apply Step 2 to $P'$.

\begin{figure}
\begin{center} 
		\begin{tikzpicture}[]
			\draw (-1,0) node {(a)};	
			\node[circle, fill, inner sep=2, label=0:$v'$] (a) at (2,0) {};
			\node[ fill, inner sep=2, label=90:$w$] (b) at (2,1) {};
			\node[circle,fill, inner sep=2, label=90:$v$] (c) at (1,1) {};
			\node[circle,fill, inner sep=2, label=90:$y$] (d) at (3,1) {};
			\node[circle,fill, inner sep=2, label=90:$x$] (e) at (0,1) {};
			\color{darkgreen}
			\node[label=90:$P_i$] (h) at (1.5,0.8) {};
			\color{red}
			\node[label=270:$Q$] (h) at (1.5,0.6) {};
			\color{black}
			\draw[thick, red] (a) to (d);
			\draw[thick, red] (a) to (c);
			\draw[thick, darkgreen] (e) to (c) to (b) to (d);
		\end{tikzpicture}
		\hspace{3cm}
		\begin{tikzpicture}[]
			\draw (-1,0) node {(b)};	
			\node[circle, fill, inner sep=2, label=0:$v'$] (a) at (2,0) {};
			\node[ fill, inner sep=2, label=90:$w$] (b) at (2,1) {};
			\node[circle,fill, inner sep=2, label=90:$v$] (c) at (1,1) {};
			\node[circle,fill, inner sep=2, label=90:$y$] (d) at (3,1) {};
			\node[circle,fill, inner sep=2, label=90:$x$] (e) at (0,1) {};
			\color{darkgreen}
			\node[label=270:$P'$] (h) at (1.45,0.6) {};
			\color{red}
			\node[label=90:$Q'$] (h) at (1.5,0.8) {};
			\color{black}
			\draw[thick, darkgreen] (a) to (d);
			\draw[thick, darkgreen] (a) to (c);
			\draw[thick,red] (c) to (b);
			\draw[thick,red] (d) to (b);
			\draw[thick, darkgreen] (c) to (e);
		\end{tikzpicture} 
\caption{\label{fig:prop567b}Proof of Lemma \ref{lemma:prop567}: Condition (E7) is violated. (a) $P_i$ (green) is a nonpendant 3-ear,
the 2-ear $Q$ (red) is the first ear attached to it. Here $w$ has degree 2 but $v'$ has larger degree. (b) The new ears $P'$ (green) and $Q'$ (red).}

\end{center}
\end{figure}
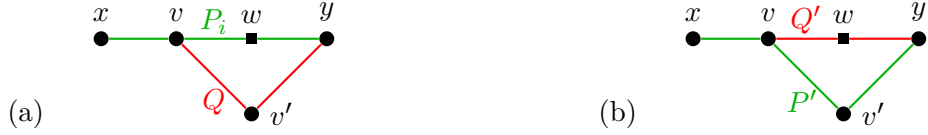
	
We now consider the running time.	
Each of Step 1, 2, and 3 takes $O(n)$ time.
As soon as (E3) is violated for some ear, we apply Lemma \ref{lemma:prop3}, which increases the number of trivial ears.
While this does not happen, after $i$ iterations of Steps 1, 2, and 3, the first $i$ ears satisfy (E5), (E6), and (E7). 
So after at most $n-1$ iterations we are either done or increase the number of trivial ears.
\end{proof}

\begin{corollary}
\label{cor:main}
Given a graph $G$ with property (P), one can compute an ear-decomposition of $G$ with properties (E1)--(E7) in $O(n^3)$ time.
\end{corollary}

\begin{proof}
We first compute an ear-decomposition with property (E1), using Proposition \ref{prop:cheriyan}.
Then we apply Lemma \ref{lemma:prop2}, \ref{lemma:prop3}, \ref{lemma:prop4}, and \ref{lemma:prop567}
until all properties are satisfied.
Lemma \ref{lemma:prop567} does not decrease the number of trivial ears, and
each application of Lemma \ref{lemma:prop2}, \ref{lemma:prop3}, or \ref{lemma:prop4} increases the number of trivial ears.
As the number of trivial ears can increase by at most $n-2$ (the number of edges in nontrivial ears is always between $n$ and $2n-2$),
the running time follows.
\end{proof}

\section{Lower bounds \label{sec:lowerbounds}}

The first two lower bounds are well-known, they are lower bounds even for 2EC.

\begin{lemma}[Cheriyan, Seb\H{o} and Szigeti \cite{cheriyan}]
\label{lemma:lbphi}
Let $G$ be a 2-connected graph. Then
\[OPT(G) \ \ge \ n-1+\varphi (G).\]
\end{lemma}

\begin{proof}
Let $H$ be a 2-connected spanning subgraph of $G$.
Any ear-decomposition of $H$ can be extended to an ear-decomposition of $G$ by adding trivial ears, so
$H$ contains at least $\varphi(G)$ even ears.
So $H$ has at least $\varphi(G)$ ears, and hence $|E(H)|\ge n-1+\varphi(G)$.
\end{proof}

\begin{lemma}[Garg, Santosh and Singla \cite{garg}]
\label{lemma:lbgarg}
Let $G$ be a 2-connected graph and $W$ a proper subset of the vertices. 
Let $q_W$ be the number of connected components of $G[W]$, the subgraph of $G$ induced by $W$ (and $q_{\emptyset}:=0$).
Then
\[OPT(G) \ \ge \ |W|+q_W.\]
\end{lemma}

\begin{proof}
Let $H$ be a 2-connected spanning subgraph of $G$.
For each connected component $C$ of $G[W]$ there are at least two edges in $H$ that connect $C$ with $V(G)\backslash W$.
Moreover, every vertex in $C$ has degree at least two in $H$, so $H$ contains at least $|V(C)|+1$ edges with at least one endpoint in $C$. 
Summing over all $C$ yields that $|E(H)|\ge  |W| + q_W$.
\end{proof}

The following lower bound applies only to 2VC, and it requires property (P) and an ear-decomposition satisfying (E1)--(E7).

\begin{lemma}
\label{lemma:lbnew}
Let $G$ be a graph with property (P) and an ear-decomposition with (E1)--(E7).
Let $k$ be the number of nonpendant 3-ears $P$ such that the first nontrivial ear attached to $P$ has length 2 or 3. Then
\[OPT(G) \ \ge \ n-1+k.\]
\end{lemma}

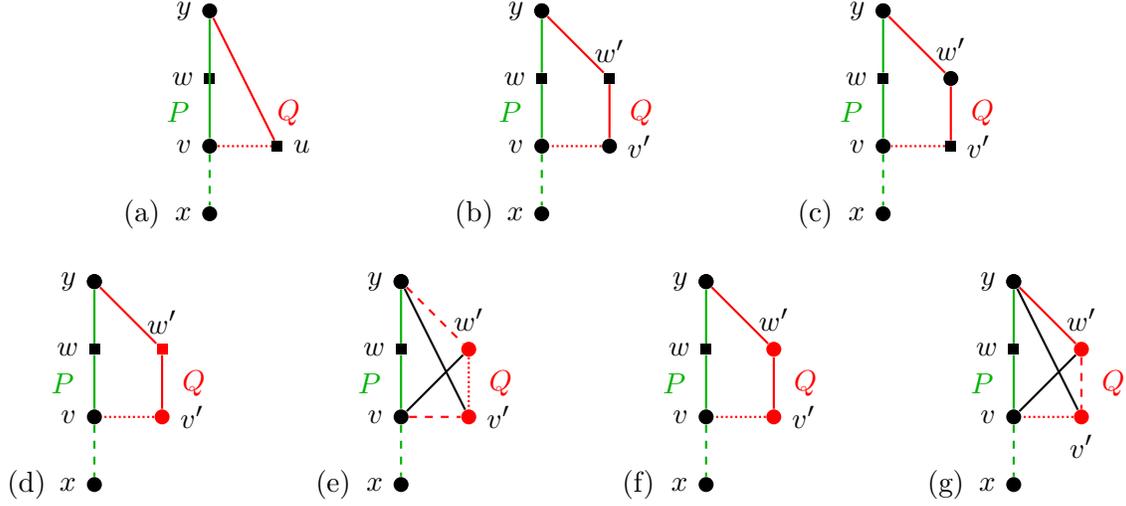
\begin{figure}
\begin{center} 		
		\hfill
		\begin{tikzpicture}[scale=0.9]
			\draw (-1,0) node {(a)};	
			\node[circle,fill, inner sep=2, label=180:$y$] (a) at (0,3) {};
			\node[fill, inner sep=2, label=180:$w$] (b) at (0,2) {};
			\node[circle,fill, inner sep=2, label=180:$v$] (c) at (0,1) {};
			\node[circle,fill, inner sep=2, label=180:$x$] (d) at (0,0) {};
			\node[fill, inner sep=2, label=0:$u$] (e) at (1,1) {};
			\color{darkgreen}
			\node[label=180:$P$] (h) at (0,1.5) {};
			\color{red}
			\node[label=0:$Q$] (h) at (0.7,1.5) {};
			\color{black}
			\draw[thick, darkgreen] (a) to (c);
			\draw[thick, dashed, darkgreen] (c) to (d);
			\draw[thick, densely dotted, red] (c) to (e);
			\draw[thick, red] (e) to (a);
		\end{tikzpicture}
		\hfill
		\begin{tikzpicture}[scale=0.9]
			\draw (-1,0) node {(b)};	
			\node[circle, fill, inner sep=2, label=180:$x$] (x) at (0,0) {};
			\node[circle, fill, inner sep=2, label=180:$v$] (v) at (0,1) {};
			\node[fill, inner sep=2, label=180:$w$] (w) at (0,2) {};
			\node[circle,fill, inner sep=2, label=180:$y$] (y) at (0,3) {};
			\node[circle,fill, inner sep=2, label=0:$v'$] (v') at (1,1) {};
			\node[fill, inner sep=2, label=90:$w'$] (w') at (1,2) {};
			\color{darkgreen}
			\node[label=180:$P$] (h) at (0,1.5) {};
			\color{red}
			\node[label=0:$Q$] (h) at (1,1.5) {};
			\color{black}
			\node[circle,fill, inner sep=2] (y') at (y) {};
			\draw[thick, dashed, darkgreen] (v) to (x);
			\draw[thick, darkgreen] (v) to (w);
			\draw[thick, red] (y) to (w');
			\draw[thick, densely dotted, red] (v') to (v);
			\draw[thick, red] (w') to (v');
			\draw[thick, darkgreen] (w) to (y');
		\end{tikzpicture}
		\hfill
		\begin{tikzpicture}[scale=0.9]
			\draw (-1,0) node {(c)};	
			\node[circle, fill, inner sep=2, label=180:$x$] (x) at (0,0) {};
			\node[circle, fill, inner sep=2, label=180:$v$] (v) at (0,1) {};
			\node[fill, inner sep=2, label=180:$w$] (w) at (0,2) {};
			\node[circle,fill, inner sep=2, label=180:$y$] (y) at (0,3) {};
			\node[fill, inner sep=2, label=0:$v'$] (v') at (1,1) {};
			\node[circle, fill, inner sep=2, label=90:$w'$] (w') at (1,2) {};
			\node[circle,fill, inner sep=2] (y') at (y) {};
			\color{darkgreen}
			\node[label=180:$P$] (h) at (0,1.5) {};
			\color{red}
			\node[label=0:$Q$] (h) at (1,1.5) {};
			\color{black}
			\draw[thick, dashed, darkgreen] (v) to (x);
			\draw[thick, darkgreen] (v) to (w);
			\draw[thick, red] (y) to (w');
			\draw[thick, densely dotted, red] (v') to (v);
			\draw[thick, red] (w') to (v');
			\draw[thick, darkgreen] (w) to (y');
		\end{tikzpicture}
		\hfill\,\\[3mm]
		\begin{tikzpicture}[scale=0.9]
			\draw (-1,0) node {(d)};	
			\node[circle, fill, inner sep=2, label=180:$x$] (x) at (0,0) {};
			\node[circle, fill, inner sep=2, label=180:$v$] (v) at (0,1) {};
			\node[fill, inner sep=2, label=180:$w$] (w) at (0,2) {};
			\node[circle,fill, inner sep=2, label=180:$y$] (y) at (0,3) {};
			\node[circle,fill=red, inner sep=2, label=0:$v'$] (v') at (1,1) {};
			\node[fill=red, inner sep=2, label=90:$w'$] (w') at (1,2) {};
			\node[circle,fill, inner sep=2] (y') at (y) {};
			\color{darkgreen}
			\node[label=180:$P$] (h) at (0,1.5) {};
			\color{red}
			\node[label=0:$Q$] (h) at (1,1.5) {};
			\color{black}
			\draw[thick, dashed, darkgreen] (v) to (x);
			\draw[thick, darkgreen] (v) to (w);
			\draw[thick, red] (y) to (w');
			\draw[thick, densely dotted, red] (v') to (v);
			\draw[thick, red] (w') to (v');
			\draw[thick, darkgreen] (w) to (y');
		\end{tikzpicture}
		\hfill
		\begin{tikzpicture}[scale=0.9]
			\draw (-1,0) node {(e)};	
			\node[circle, fill, inner sep=2, label=180:$x$] (x) at (0,0) {};
			\node[circle, fill, inner sep=2, label=180:$v$] (v) at (0,1) {};
			\node[fill, inner sep=2, label=180:$w$] (w) at (0,2) {};
			\node[circle,fill, inner sep=2, label=180:$y$] (y) at (0,3) {};
			\node[circle,fill=red, inner sep=2, label=0:$v'$] (v') at (1,1) {};
			\node[circle, fill=red, inner sep=2, label=90:$w'$] (w') at (1,2) {};
			\node[circle,fill, inner sep=2] (y') at (y) {};
			\color{darkgreen}
			\node[label=180:$P$] (h) at (0,1.5) {};
			\color{red}
			\node[label=0:$Q$] (h) at (1,1.5) {};
			\color{black}
			\draw[thick, dashed, darkgreen] (v) to (x);
			\draw[thick, darkgreen] (v) to (w);
			\draw[thick] (y) to (v');
			\draw[thick] (w') to (v);
			\draw[thick, densely dotted, red] (w') to (v');
			\draw[thick, darkgreen] (w) to (y');
			\draw[thick,dashed, red] (v) to (v');
			\draw[thick,dashed, red] (y) to (w');
		\end{tikzpicture}
		\hfill
		\begin{tikzpicture}[scale=0.9]
			\draw (-1,0) node {(f)};	
			\node[circle, fill, inner sep=2, label=180:$x$] (x) at (0,0) {};
			\node[circle, fill, inner sep=2, label=180:$v$] (v) at (0,1) {};
			\node[fill, inner sep=2, label=180:$w$] (w) at (0,2) {};
			\node[circle,fill, inner sep=2, label=180:$y$] (y) at (0,3) {};
			\node[circle,fill=red, inner sep=2, label=0:$v'$] (v') at (1,1) {};
			\node[circle, fill=red, inner sep=2, label=90:$w'$] (w') at (1,2) {};
			\node[circle,fill, inner sep=2] (y') at (y) {};
			\color{darkgreen}
			\node[label=180:$P$] (h) at (0,1.5) {};
			\color{red}
			\node[label=0:$Q$] (h) at (1,1.5) {};
			\color{black}
			\draw[thick, dashed, darkgreen] (v) to (x);
			\draw[thick, darkgreen] (v) to (w);
			\draw[thick, red] (y) to (w');
			\draw[thick, densely dotted, red] (v') to (v);
			\draw[thick, red] (w') to (v');
			\draw[thick, darkgreen] (w) to (y');
		\end{tikzpicture}
		\hfill
		\begin{tikzpicture}[scale=0.9]
			\draw (-1,0) node {(g)};	
			\node[circle, fill, inner sep=2, label=180:$x$] (x) at (0,0) {};
			\node[circle, fill, inner sep=2, label=180:$v$] (v) at (0,1) {};
			\node[fill, inner sep=2, label=180:$w$] (w) at (0,2) {};
			\node[circle,fill, inner sep=2, label=180:$y$] (y) at (0,3) {};
			\node[circle, fill=red, inner sep=2, label=270:$v'$] (v') at (1,1) {};
			\node[circle, fill=red, inner sep=2, label=90:$w'$] (w') at (1,2) {};
			\color{darkgreen}
			\node[label=180:$P$] (h) at (0,1.5) {};
			\color{red}
			\node[label=0:$Q$] (h) at (1,1.5) {};
			\color{black}
			\node[circle,fill, inner sep=2] (y') at (y) {};
			\draw[thick, dashed, darkgreen] (v) to (x);
			\draw[thick, darkgreen] (v) to (w);
			\draw[thick, red] (y) to (w');
			\draw[thick, red, densely dotted] (v') to (v);
			\draw[thick, dashed, red] (w') to (v');
			\draw[thick, darkgreen] (w) to (y');
			\draw[thick] (w') to (v);
			\draw[thick] (v') to (y);
		\end{tikzpicture}
		
\caption{\label{fig:lb}Proof of Lemma \ref{lemma:lbnew}: 
a nonpendant 3-ear $P$ (green) and the first nontrivial ear $Q$ attached to it (red).
Dotted and solid edges belong to $H$; the dotted edge is deleted.
Again, inner vertices of pendant ears are red, and vertices shown as squares have no other incident edges than shown.}

\end{center}
\end{figure}

\begin{proof}
Let $H$ be a 2-connected spanning subgraph of $G$ with minimum number of edges.
We will remove $k$ edges from $H$ and still maintain a connected graph. 

We scan the ears of our ear-decomposition of $G$ in reverse order.
If the current ear $P$ is a nonpendant 3-ear such that the first nontrivial ear $Q$ attached to $P$ has length 2 or 3,
we delete an edge from $H$ that is chosen as follows.

Let $E(P)=\{\{x, v\}, \{v,w\}, \{w, y\}\}$, and let $Q$ have endpoints $v$ and $y$ (cf.\ condition (E3)).

\case{Case 1} $Q$ is a 2-ear.

Then $w$ (by (E6)) and the inner vertex $u$ of $Q$ (by (E7)) have degree 2 in $G$.
Therefore $H$ contains all four edges $\{v,w\}, \{w,y\}, \{y,u\},$ and $\{u,v\}$. 
Deleting the edge $\{u,v\}$ will therefore not disconnect the graph (Figure \ref{fig:lb}(a)).

\case{Case 2} $Q$ is a nonpendant 3-ear.

Let $E(Q)=\{\{v, v'\}, \{v',w'\}, \{w', y\}\}$. 
By (E6), $w$ has degree 2 in $G$.
Since $Q$ is also a nonpendant 3-ear, (E6) also implies that $v'$ or $w'$ has degree 2 in $G$.
If $w'$ has degree 2 in $G$ (Figure \ref{fig:lb}(b)), then by property (P) the edge $\{v,v'\}$ is not redundant.
If $v'$ has degree 2 in $G$ (Figure \ref{fig:lb}(c)), then by property (P) the edge $\{w',y\}$ is not redundant.
In both cases, $H$ contains all five edges $\{v,w\}, \{w,y\}, \{y,w'\},\{w',v'\},$ and $\{v',v\}$.
Deleting $\{v',v\}$ will therefore not disconnect the graph.

\case{Case 3} $Q$ is a pendant 3-ear.

Let $E(Q)=\{\{v, v'\}, \{v',w'\}, \{w', y\}\}$. 
Again, $w$ has degree 2 in $G$. By (E5) and (E4), there are the following two cases.

\case{Case 3.1} $w'$ has degree 2 in $G$.

Then, again, by property (P) the edge $\{v,v'\}$ is not redundant,
so $H$ contains all five edges
$\{v,w\},\{w,y\},\{y,w'\},\{w',v'\},$ and $\{v',v\}$.
Deleting $\{v',v\}$ will therefore not disconnect the graph (Figure \ref{fig:lb}(d)).

\case{Case 3.2} $\Gamma (w')=\{v', y, v\}$ and $\Gamma (v')\subseteq\{y,v,w'\}$:.

As $H$ contains at least two edges incident to $v'$ and at least two edges incident to $w'$,
there are the following three subcases.

	\bfseries Case 3.2.1: \mdseries $H$ contains $\{v,w'\}, \{w',v'\},$ and $\{v',y\}$ (Figure \ref{fig:lb}(e)). 

	\bfseries Case 3.2.2: \mdseries $H$ contains $\{v,v'\}, \{v',w'\},$ and $\{w', y\}$ (Figure \ref{fig:lb}(f)). 

	\bfseries Case 3.2.3: \mdseries $H$ contains $\{v,w'\}, \{w',y\}, \{y,v'\},$ and $\{v',v\}$ (Figure \ref{fig:lb}(g)).
		
In each of the three subcases, deleting one of these edges will not disconnect the graph.

\bigskip		
In each case we delete an edge from $Q$. Note that no edge from $P$ or $Q$ was deleted in any previous step,
because these ears were never considered as $Q$ before.

We conclude that after deleting $k$ edges from $H$ we still have a connected graph. Hence at least $n-1$ edges remain,
so $H$ had at least $n-1+k$ edges.
\end{proof}

We remark that this is not a lower bound for 2EC, as the example in Figure \ref{fig:no2eclb} shows.

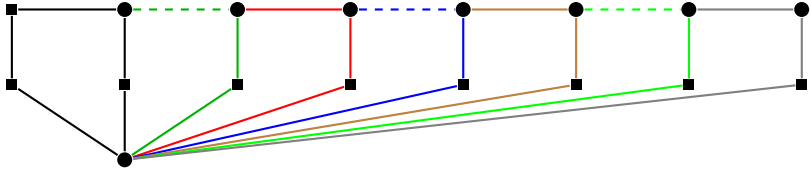
\begin{figure}
\begin{center} 
		\begin{tikzpicture}[thick,xscale=1.5]
			\node[circle, fill, inner sep=2] (a1) at (1,0) {};
			\node[ fill, inner sep=2] (a2) at (0,1) {};
			\node[ fill, inner sep=2] (a3) at (0,2) {};
			\node[circle, fill, inner sep=2] (a4) at (1,2) {};
			\node[fill, inner sep=2] (a5) at (1,1) {};
			\node[circle, fill, inner sep=2] (b1) at (2,2) {};
			\node[fill, inner sep=2] (b2) at (2,1) {};
			\node[circle, fill, inner sep=2] (c1) at (3,2) {};
			\node[fill, inner sep=2] (c2) at (3,1) {};
			\node[circle, fill, inner sep=2] (d1) at (4,2) {};
			\node[ fill, inner sep=2] (d2) at (4,1) {};
			\node[circle, fill, inner sep=2] (e1) at (5,2) {};
			\node[ fill, inner sep=2] (e2) at (5,1) {};
			\node[circle, fill, inner sep=2] (f1) at (6,2) {};
			\node[ fill, inner sep=2] (f2) at (6,1) {};
			\node[circle, fill, inner sep=2] (g1) at (7,2) {};
			\node[ fill, inner sep=2] (g2) at (7,1) {};
			\draw[] (a1) to (a2) to (a3) to (a4) to (a5) to (a1);
			\draw[darkgreen, dashed] (a4) to (b1); \draw[darkgreen] (b1) to (b2) to (a1);
			\draw[red] (b1) to (c1) to (c2) to (a1);
			\draw[blue, dashed] (c1) to (d1); \draw[blue] (d1) to (d2) to (a1);
			\draw[brown] (d1) to (e1) to (e2) to (a1);
			\draw[green, dashed] (e1) to (f1); \draw[green] (f1) to (f2) to (a1);
			\draw[gray] (f1) to (g1) to (g2) to (a1);
		\end{tikzpicture}
\caption{\label{fig:no2eclb} A 2-connected graph without redundant edges and an ear-decomposition
with (E1)--(E7), consisting of a closed 5-ear (left) and six 3-ears (each with a different color, the one on the right is pendant);
The lower bound from Lemma \ref{lemma:lbnew} is $n-1+5=21$, but the 20 solid edges
form a 2-edge-connected spanning subgraph.}

\end{center}
\end{figure}

\section{The approximation ratio}

We will now show that the nontrivial ears in an ear-decomposition with (E1)--(E7) have at most $\frac{10}{7}\opt(G)$ edges.
In the proof we need to solve the following simple optimization problem.

\begin{lemma}
\label{lemma:107}
Let $f:\mathbb{R}^7\rightarrow\mathbb{R}$ with $f(a,b,c,d,e,n,\varphi)=\frac{\frac{5}{4}(n-1)+\frac{3}{4}\varphi +\frac{1}{2}(a+b+c+e)}{\max \{n-1+\varphi, 3a+4b+2c+2d+2e, n-1+b+c\}}$.
Then
$$\max \bigl\{ f(a,b,c,d,e,n,\varphi) \ : \ a, b, c, d, e, \varphi \ge 0,\, n\ge 2,\, 2a+3b+2c+5d+6e\le n-1 \bigr\} \ = \ \frac{10}{7}.$$
\end{lemma}

\begin{proof}
For $n=15$, $a=4$, $e=1$, and $b=c=d=\varphi =0$
we have $f(a,b,c,d,e,n,\varphi)=\frac{10}{7}$.

To show the upper bound, 
let $a,b,c,d,e,\varphi\ge 0$ and $n\ge 2$ with $2a+3b+2c+5d+6e\le n-1$. 
We show $f(a,b,c,d,e,n,\varphi )\le\frac{10}{7}$.
Let 
$$k \ := \ \max\{3a+4b+2c+2d+2e,\, n+\varphi -1,\, n+b+c-1\}$$
be the denominator. 
Then $n-1+\varphi\le k$ and thus
$$f(a,b,c,d,e,n,\varphi ) \ \le \ \frac{\frac{1}{2}(n-1)+\frac{3}{4}k+\frac{1}{2}(a+b+c+e)}{k}.$$

\bigskip

Consider the LP $P_{n,k}$ with variables $a,b,c,d,e$:
$$\max \bigl\{a+b+c+e : a,b,c,d,e\ge 0,\, n-1+b+c\le k,\, 3a+4b+2c+2d+2e\le k,\, 2a+3b+2c+5d+6e\le n-1 \bigr\}.$$
The dual LP $D_{n,k}$, with variables $x,y,z$ is:
\begin{align*}
\min \bigl\{(k-n+1)x + ky + (n-1)z &\ : \ x,y,z\ge 0,\, 3y+2z\ge 1,\, x+4y+3z\ge 1,\, \\
& \hspace*{5mm}  x+2y+2z\ge 1,\, 2y+5z\ge 0,\, 2y+6z\ge 1 \bigr\}.
\end{align*}

Since $(x,y,z)=(\frac{2}{7},\frac{2}{7},\frac{1}{14})$ is a feasible solution of the dual LP with value $\frac{4}{7} k - \frac{3}{14} (n-1)$, 
the primal LP value is also at most $\frac{4}{7} k - \frac{3}{14} (n-1)$, and we get  
$$f(a,b,c,d,e,n,\varphi ) \ \le \ \frac{\frac{1}{2}(n-1)+\frac{3}{4}k + \frac{4}{14}k - \frac{3}{28}(n-1)}{k} 
\ = \ \frac{11(n-1)+29k}{28k}
\ \le \ \frac{40k}{28k} \ = \ \frac{10}{7},$$
because $n-1\le k$.
\end{proof}

\begin{theorem}
There is a $\frac{10}{7}$-approximation algorithm for 2VC with running time $O(n^3)$.
\end{theorem}

\begin{proof}
Let $G$ be a given 2-connected graph.
Due to Corollary \ref{cor:redundant} we can delete some redundant edges in $O(n^3)$ time so that 
the resulting spanning subgraph $\bar G$ of $G$ satisfies property (P) 
and $\opt(G)=\opt(\bar G)$.
Then we use Corollary \ref{cor:main} to compute an ear-decomposition of $\bar G$ with (E1)--(E7) in $O(n^3)$ time.
We delete the trivial ears and output the resulting 2-connected spanning subgraph $H$. 
We show that $H$ has at most $\frac{10}{7}\opt(G)$ edges.

Let $a$ be the number of pendant 3-ears, and let 
$b$, $c$, $d$, and $e$ denote the number of nonpendant 3-ears $P$ such that the first nontrivial ear attached to $P$ 
has length 2, 3, 4, and at least 5, respectively. (See again Figure \ref{fig:3earcases}.)
We claim that
\begin{equation}
\tag{\ensuremath{\star}}\label{eq:upperbound}
|E(H)| \ \le \ \frac{5}{4}(n-1)+\frac{3}{4}\varphi (\bar G)+\frac{1}{2}(a+b+c+e)
\end{equation}

To show this, we sum over all ears, distinguishing cases as follows.
For a 3-ear $P$ whose first attached ear $Q$ is a 4-ear, we have
$|E(P)|+|E(Q)| = 7 = \frac{5}{4} (|\inn(P)|+|\inn(Q)|) + \frac{3}{4}$. 
Note that by (E3) no ear can be the first nontrivial attached ear of two 3-ears, so no 4-ear is counted twice here.
For any other 3-ear $P$ we have $|E(P)|=3 = \frac{5}{4}|\inn(P)|+\frac{1}{2}$.
For any other even ear $P$ we have $|E(P)| \le \frac{5}{4}|\inn(P)|+\frac{3}{4}$.
Finally, for any odd ear $P$ of length at least 5 we have $|E(P)| \le \frac{5}{4}|\inn(P)|$.
Summation yields (\ref{eq:upperbound}), because by (E1) there are $\varphi(\bar G)$ even ears.

Now we compare the upper bound in (\ref{eq:upperbound}) to the lower bounds of the previous section.
First, by Lemma \ref{lemma:lbphi}, $\opt(\bar G) \ge n-1+\varphi (\bar G)$.

Let $W_1$ be the set of inner vertices of pendant 3-ears, $W_2$ the set of inner vertices of 3-ears that have degree 2 and $W_3$ the set of inner vertices of
2-ears that are the first nontrivial ear attached to a 3-ear. Let $W:=W_1\cup W_2\cup W_3$. By (E6) we have $|W|=2a+2b+c+d+e$.
Then consider $\bar G[W]$, the subgraph of $\bar G$ induced by $W$.
By (E2) and (E7), this subgraph has no edges between inner vertices of different ears.
Hence it has $q_W=a+2b+c+d+e$ connected components.
By Lemma \ref{lemma:lbgarg} we have $\opt(\bar G) \ge |W|+q_W = 3a+4b+2c+2d+2e$.

Finally, by Lemma \ref{lemma:lbnew}, we have $\opt(\bar G) \ge n-1+b+c$.	
Together we conclude
\begin{equation}
\tag{\ensuremath{\star\star}}\label{eq:lowerbound}
\opt(\bar G) \ \ge \  \max \bigl\{n-1+\varphi(\bar G),\, 3a+4b+2c+2d+2e,\, n-1+b+c \bigr\}.
\end{equation}
Moreover, $2a+3b+2c+5d+6e\le n-1$, by summing over the 3-ears, taking each together with the first nontrivial ear attached to it unless it is a 3-ear itself
(again, by (E3), no nontrivial ear can be the first attached ear of two 3-ears).

From (\ref{eq:upperbound}), Lemma \ref{lemma:107}, and (\ref{eq:lowerbound}) we get
\begin{eqnarray*}
|E(H)|  &\le& \frac{5}{4}(n-1) + \frac{3}{4}\varphi (\bar G) + \frac{1}{2}(a+b+c+e) \\
&\le& \frac{10}{7} \max \bigl\{n-1+\varphi(\bar G),\, 3a+4b+2c+2d+2e,\, n-1+b+c \bigr\} \\
&\le& \frac{10}{7} \opt(\bar G) \\ 
&=& \frac{10}{7} \opt(G).
\end{eqnarray*}
\end{proof}

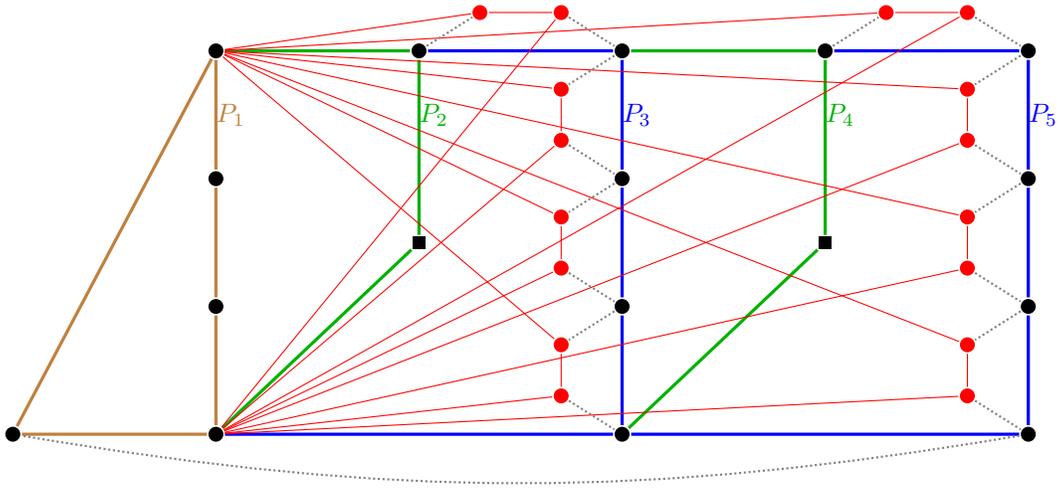
\begin{figure}
\begin{center}		
\small
		\begin{tikzpicture}[very thick,xscale=2.7, yscale =1.7]
			\color{brown}
			\node[label=0:$P_1$] at (0.9,2.5) {};
			\color{darkgreen}
			\node[label=0:$P_2$] at (1.9,2.5) {};
			\node[label=0:$P_4$] at (3.9,2.5) {};
			\color{blue}
			\node[label=0:$P_3$] at (2.9,2.5) {};
			\node[label=0:$P_5$] at (4.9,2.5) {};
			\color{black}
			\node[circle, fill, inner sep=2] (a) at (0,0) {};
			\node[circle, fill, inner sep=2] (a2) at (1,3) {};
			\node[circle, fill, inner sep=2] (a3) at (1,2) {};
			\node[circle, fill, inner sep=2] (a4) at (1,1) {};
			\node[circle, fill, inner sep=2] (a5) at (1,0) {};
			\node[fill, inner sep=2.3] (b1) at (2,1.5) {};
			\node[circle, fill, inner sep=2] (b2) at (2,3) {};
			\node[circle, fill, inner sep=2] (b3) at (3,0) {};
			\node[circle, fill, inner sep=2] (b4) at (3,1) {};
			\node[circle, fill, inner sep=2] (b5) at (3,2) {};
			\node[circle, fill, inner sep=2] (b6) at (3,3) {};
			\draw[brown] (a) to (a2) to (a3) to (a4) to (a5) to (a);
			\draw[darkgreen] (a5) to (b1) to (b2) to (a2);
			\draw[blue] (a5) to (b3) to (b4) to (b5) to (b6) to (b2);
			\node[fill, inner sep=2.3] (c1) at (4,1.5) {};
			\node[circle, fill, inner sep=2] (c2) at (4,3) {};
			\node[circle, fill, inner sep=2] (c3) at (5,0) {};
			\node[circle, fill, inner sep=2] (c4) at (5,1) {};
			\node[circle, fill, inner sep=2] (c5) at (5,2) {};
			\node[circle, fill, inner sep=2] (c6) at (5,3) {};
			\draw[darkgreen] (b3) to (c1) to (c2) to (b6);
			\draw[blue] (b3) to (c3) to (c4) to (c5) to (c6) to (c2);
			\node[circle, fill, inner sep=2, red] (d1) at (2.3,3.3) {};
			\node[circle, fill, inner sep=2, red] (d2) at (2.7, 3.3) {};
			\node[circle, fill, inner sep=2, red] (e1) at (2.7,2.7) {};
			\node[circle, fill, inner sep=2, red] (e2) at (2.7, 2.3) {};
			\node[circle, fill, inner sep=2, red] (f1) at (2.7,1.7) {};
			\node[circle, fill, inner sep=2, red] (f2) at (2.7, 1.3) {};
			\node[circle, fill, inner sep=2, red] (g1) at (2.7,0.7) {};
			\node[circle, fill, inner sep=2, red] (g2) at (2.7,0.3) {};
			\draw[red,thin] (a2) to (d1) to (d2) to (a5);
			\draw[red,thin] (a2) to (e1) to (e2) to (a5);
			\draw[red,thin] (a2) to (f1) to (f2) to (a5);
			\draw[red,thin] (a2) to (g1) to (g2) to (a5);
			\node[circle, fill, inner sep=2, red] (h1) at (4.3, 3.3) {};
			\node[circle, fill, inner sep=2, red] (h2) at (4.7, 3.3) {};
			\node[circle, fill, inner sep=2, red] (i1) at (4.7, 2.7) {};
			\node[circle, fill, inner sep=2, red] (i2) at (4.7, 2.3) {};
			\node[circle, fill, inner sep=2, red] (j1) at (4.7, 1.7) {};
			\node[circle, fill, inner sep=2, red] (j2) at (4.7, 1.3) {};
			\node[circle, fill, inner sep=2, red] (k1) at (4.7, 0.7) {};
			\node[circle, fill, inner sep=2, red] (k2) at (4.7,0.3) {};
			\draw[red,thin] (a2) to (h1) to (h2) to (a5);
			\draw[red,thin] (a2) to (i1) to (i2) to (a5);
			\draw[red,thin] (a2) to (j1) to (j2) to (a5);
			\draw[red,thin] (a2) to (k1) to (k2) to (a5);
			\draw[gray,thick, densely dotted] (b2) to (d1); \draw[gray,thick, densely dotted] (d2) to (b6) to (e1); \draw[gray,thick, densely dotted] (e2) to (b5);
			\draw[gray,thick, densely dotted] (b5) to (f1); \draw[gray,thick, densely dotted] (f2) to (b4) to (g1); \draw[gray,thick, densely dotted] (g2) to (b3); 
			\draw[gray,thick, densely dotted] (c2) to (h1); \draw[gray,thick, densely dotted] (h2) to (c6) to (i1); \draw[gray,thick, densely dotted] (i2) to (c5) to (j1); 
			\draw[gray,thick, densely dotted] (j2) to (c4) to (k1); \draw[gray,thick, densely dotted] (k2) to (c3);
			\draw[gray,thick, densely dotted,bend left=15] (c3) to (a);
		\end{tikzpicture}

\caption{\label{fig:tightex}A tight example.
		For each $k\in\mathbb{N}$ there is a graph $G_k$ with $14k+5$ vertices and property (P) and
		an ear-decomposition of $G_k$ as follows (we show $k=2$). There is a closed 5-ear $P_1$, for $i=1,\ldots,k$ 
		a 3-ear $P_{2i}$ and a 5-ear $P_{2i+1}$, then $4k$ pendant 3-ears (red), and $8k+1$ trivial ears (dotted) 
		such that (E1)--(E7) are satisfied. Deleting the trivial ears yields a minimal 2-connected spanning subgraph with $20k+5$ edges. 
		However, $G_k$ has a Hamiltonian circuit (using all the dotted edges), so $\opt(G_k)=|V(G_k)|=14k+5$.}
\end{center}
\end{figure}

Our analysis is tight; see Figure \ref{fig:tightex}.
We also remark that for graphs with minimum degree 3 we immediately get the better approximation ratio $\frac{17}{12}$ (as in \cite{cheriyan}).
	
\section*{Appendix: counterexample to the Vempala-Vetta approach}

We will now show that the approach of Vempala and Vetta \cite{vempala}, who claimed to have a $\frac{4}{3}$-approximation algorithm, does not work
without modifications. As lower bound they use 
the minimum number $L_{D2}$ of edges of a spanning subgraph in which every vertex has degree at least two.
Moreover, in a 2-connected graph $G$, they define a \emph{beta} (vertex or pair) to be a set $C\subseteq V(G)$ with $|C|\le 2$ 
that is one of at least three connected components of $G-\{s,t\}$ for some $s,t\in V(G)$.

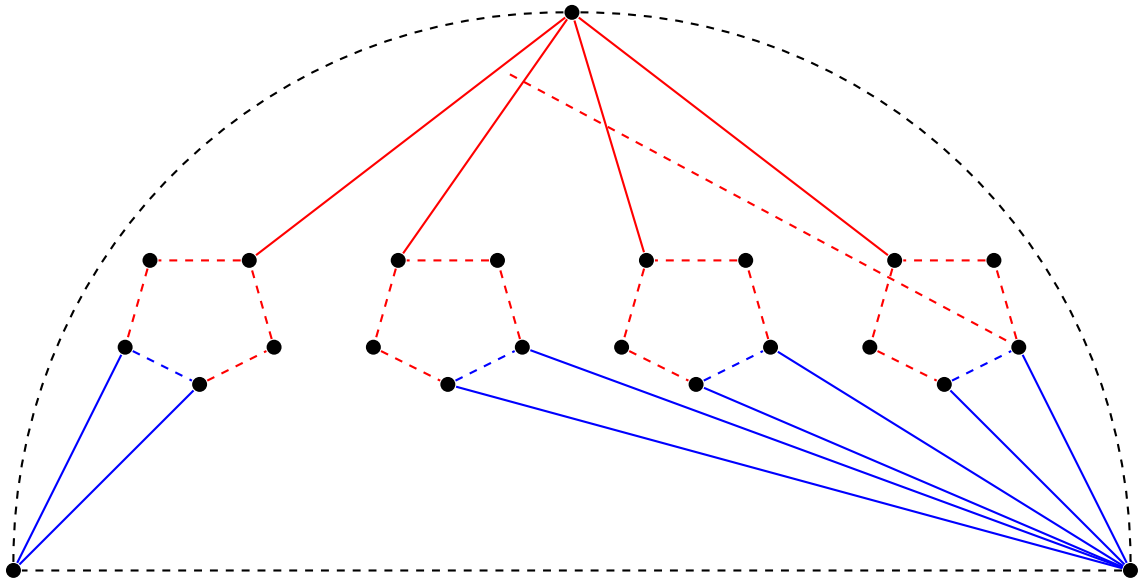
\begin{figure}
\begin{center}
	\begin{tikzpicture}[scale=1.65,thick]
		\node[circle,fill, inner sep=2] (a1) at (0.5,0.5) {};
		\node[circle,fill, inner sep=2] (a2) at (5,-4) {};
		\node[circle,fill, inner sep=2] (a3) at (-4,-4) {};
		\node[circle,fill, inner sep=2] (b1) at (-2.9,-1.5) {};
		\node[circle,fill, inner sep=2] (b2) at (-3.1,-2.2) {};
		\node[circle,fill, inner sep=2] (b3) at (-2.5,-2.5) {};
		\node[circle,fill, inner sep=2] (b4) at (-1.9,-2.2) {};
		\node[circle,fill, inner sep=2] (b5) at (-2.1,-1.5) {};
		\node[circle,fill, inner sep=2] (c1) at (-0.9,-1.5) {};
		\node[circle,fill, inner sep=2] (c2) at (-1.1,-2.2) {};
		\node[circle,fill, inner sep=2] (c3) at (-0.5,-2.5) {};
		\node[circle,fill, inner sep=2] (c4) at (0.1,-2.2) {};
		\node[circle,fill, inner sep=2] (c5) at (-0.1,-1.5) {};
		\node[circle,fill, inner sep=2] (d1) at (1.1,-1.5) {};
		\node[circle,fill, inner sep=2] (d2) at (0.9,-2.2) {};
		\node[circle,fill, inner sep=2] (d3) at (1.5,-2.5) {};
		\node[circle,fill, inner sep=2] (d4) at (2.1,-2.2) {};
		\node[circle,fill, inner sep=2] (d5) at (1.9,-1.5) {};
		\node[circle,fill, inner sep=2] (e1) at (3.1,-1.5) {};
		\node[circle,fill, inner sep=2] (e2) at (2.9,-2.2) {};
		\node[circle,fill, inner sep=2] (e3) at (3.5,-2.5) {};
		\node[circle,fill, inner sep=2] (e4) at (4.1,-2.2) {};
		\node[circle,fill, inner sep=2] (e5) at (3.9,-1.5) {};
		\draw[dashed] (a2) to (a3);
		\draw[dashed] (a1) to [bend left=45] (a2);
		\draw[dashed] (a3) to [bend left=45] (a1);
		\draw[dashed, blue] (b2) to (b3); \draw[red, dashed] (b3) to (b4) to (b5) to (b1) to (b2);
		\draw[dashed, red] (c1) to (c2) to (c3); \draw[dashed, blue] (c3) to (c4); \draw[red, dashed]  (c4) to (c5) to (c1);
		\draw[dashed, red] (d1) to (d2) to (d3); \draw[dashed, blue] (d3) to (d4); \draw[red, dashed] (d4) to (d5) to (d1);
		\draw[dashed, red] (e1) to (e2) to (e3); \draw[dashed, blue] (e3) to (e4); \draw[red, dashed]  to (e4) to (e5) to (e1);
		\draw[red] (b5) to (a1);
		\draw[blue] (b2) to (a3);
		\draw[blue] (b3) to (a3);
		\draw[red] (c1) to (a1);
		\draw[blue] (c3) to (a2);
		\draw[blue] (c4) to (a2);
		\draw[red] (d1) to (a1);
		\draw[blue] (d3) to (a2);
		\draw[blue] (d4) to (a2);
		\draw[red] (e1) to (a1);
		\draw[blue] (e3) to (a2);
		\draw[blue] (e4) to (a2);
	\end{tikzpicture}

\caption{\label{fig:counterexvv}Counterexample to \cite{vempala}.
For every $k\in\mathbb{N}$ there is a 2-connected graph $G_k$ with $5k+3$ vertices and the following properties (we show $k=4$).
There is no beta and there are no adjacent degree-2 vertices. The dashed edges form a 2-regular subgraph, so $L_{D2}=n=5k+3$.
Every 2-connected spanning subgraph must contain all red edges and two out of the three blue edges incident to each dashed pentagon.
So $\opt(G_k)\ge 7k$. For $k>12$ the ratio is worse than $\frac{4}{3}$.}

\end{center}
\end{figure}

If $G$ is a 2-connected graph without adjacent degree-2 vertices and without beta, they claim that their algorithm
finds a 2-connected spanning subgraph with at most $\frac{4}{3}L_{D2}$ edges.
However, such a subgraph does not always exist; see Figure \ref{fig:counterexvv}.


\begin{thebibliography}{XX}
		\bibitem{cheriyan} J. Cheriyan, A. Seb\H{o} and Z. Szigeti: \itshape Improving on the 1.5-approximation of a smallest 2-edge connected spanning subgraph, \upshape SIAM Journal on Discrete Mathematics 14 (2001), 170-180.
		\bibitem{thuri} J. Cheriyan and R. Thurimella: \itshape Approximating minimum-size k-connected spanning subgraphs via matching, \upshape SIAM J. Comput. 30 (2000), 528-560.
		\bibitem{chong} K. W. Chong and T. W. Lam: \itshape Improving biconnectivity approximation via local optimization, \upshape Proceedings of the Seventh Annual ACM-SIAM Symposium on Discrete Algorithms (Atlanta, GA, 1996), 26–35, ACM, New York, 1996.
		\bibitem{czumaj} A. Czumaj and A. Lingas: \itshape On approximability of the minimum-cost k-connected spanning subgraph problem, \upshape Proceedings of the Tenth Annual ACM-SIAM Symposium on Discrete Algorithms (Baltimore, MD, 1999), 281–290, ACM, New York, 1999.
		\bibitem{frank} A. Frank: \itshape Conservative weightings and ear-decompositions of graphs, \upshape Combinatorica \upshape 13 \mdseries (1993), 65-81.
		\bibitem{garg} N. Garg, V. S. Santosh, and A. Singla: \itshape Improved approximation algorithms for biconnected subgraphs via better lower bounding techniques, \upshape Proceedings of the Fourth Annual ACM-SIAM Symposium on Discrete Algorithms (Austin, TX, 1993), 103–111, ACM, New York, 1993.
		\bibitem{gub} P. Gubbala: \itshape Problems in graph connectivity, \upshape Ph.D. Thesis, University of Texas at Dallas, 2006.
		\bibitem{gubbala} P. Gubbala, B. Raghavarachi: \itshape Approximation algorithms for the minimum cardinality two-connected spanning subgraph problem, \upshape Integer programming and combinatorial optimization, 422–436, Lecture Notes in Comput. Sci., 3509, Springer, Berlin, 2005.
		\bibitem{jothi} R. Jothi, B. Raghavarachi, S. Varadarajan: \itshape A 5/4-approximation algorithm for minimum 2-edge connectivity, \upshape Proceedings of the Fourteenth Annual ACM-SIAM Symposium on Discrete Algorithms (Baltimore, MD, 2003),  725–734, ACM, New York, 2003.
		\bibitem{khuller} S. Khuller and U. Vishkin: \itshape Biconnectivity approximations and graph carvings, \upshape J. Assoc. Comput. Mach. \upshape 41 \mdseries (1994), 214–235.
		\bibitem{nutov} Z. Nutov: \itshape The k-connected subgraph problem, \upshape to appear in: Handbook of Approximation Algorithms and Metaheuristics (T. Gonzalez), \upshape 2nd edition.
		\bibitem{vygen} A. Seb\H{o} and J. Vygen: \itshape Shorter tours by nicer ears: $\frac{7}{5}$-Approximation for the graph-TSP, $\frac{3}{2}$ for the path version and $\frac{4}{3}$ for two-edge-connected subgraphs, \upshape Combinatorica \upshape 34 \mdseries (2014), 597-629.
		\bibitem{vempala} S. Vempala and A. Vetta: \itshape Factor 4/3 approximation for minimum 2-connected subgraphs, \upshape Approximation algorithms for combinatorial optimization (Saarbrücken, 2000),  262–273, Lecture Notes in Comput. Sci., 1913, Springer, Berlin, 2000.
		\bibitem{whitney} H. Whitney: \itshape Nonseparable and planar graphs, \upshape Transactions of the American Mathematical Society, \upshape 34 \mdseries (1932), 339–362.
	\end{thebibliography}
\end{document}